\documentclass{article}

\usepackage{hyperref}
\usepackage{amsmath}
\usepackage{amsthm}
\usepackage{amssymb}
\usepackage{amsmath,amssymb,amsfonts,amsbsy}
\usepackage[top=2.8cm,bottom=2.0cm,left=2.5cm,right=1.8cm]{geometry}
\usepackage{color}  
\usepackage{amsthm} 

\usepackage{authblk}

\newtheorem{thm}{Theorem}[section]

\newtheorem{lem}[thm]{Lemma}
\newtheorem{definition}[thm]{Definition}
\newtheorem{proposition}[thm]{Proposition}

\newtheorem{rmk}[thm]{Remark}

\title{Explicit evaluation of the Stokes matrices for certain quantum confluent hypergeometric equations}

\author{Jinghong Lin and Xiaomeng Xu}

\begin{document}

\date{}

\newcommand{\Addresses}{{
  \bigskip
  \footnotesize
\noindent \textsc{School of Mathematical Sciences, Peking University, Beijing 100871, China}\par\nopagebreak
  \textit{E-mail address}: \texttt{linjinghong@pku.edu.cn}
}\\
\\
\footnotesize
\noindent \textsc{
School of Mathematical Sciences \& Beijing International Center
for Mathematical Research, Peking University, Beijing 100871, China}\par\nopagebreak
  \textit{E-mail address}: \texttt{xxu@bicmr.pku.edu.cn}
}

\maketitle              



\begin{abstract}
In this paper, we compute the Stokes matrices of a special quantum confluent hypergeometric system with Poincar\'e rank one. The sources of the interests in the Stokes phenomenon of such system are from representation theory and the theory of isomonodromy deformation.
\end{abstract}

\section{Introduction}

Let us take the Lie algebra ${\mathfrak {gl}_n}$ over the field of complex numbers, and its universal enveloping algebra $U\left({\mathfrak {gl}}_n\right)$ generated by $\{e_{ij}\}_{1\le i,j\le n}$ subject to the relation $[e_{ij},e_{kl}]=\delta_{jk}e_{il}-\delta_{li} e_{kj}$. Let $T=\left(T_{ij}\right)$ be the $n\times n$ matrix with entries valued in $U\left({\mathfrak {gl}}_n\right)$
\begin{eqnarray*}
T_{ij}=e_{ij}, \ \ \ \ \ \text{for} \ 1\le i,j\le n.
\end{eqnarray*}
Given any finite-dimensional irreducible representation $L\left(\lambda\right)$ of $\mathfrak{gl}_n$ with a highest weight $\lambda$, the associated quantum confluent hypergeometric system,
for $F\left(z\right)\in {\rm End}\left(L\left(\lambda\right)\right)\otimes {\rm End}\left(\mathbb{C}^n\right)$ (an $n\times n$ matrix function with entries in ${\rm End}\left(L\left(\lambda\right)\right)$), is
\begin{eqnarray}\label{chqeq}
\frac{\mathrm{d}F}{\mathrm{d}z}=h\left(\mathrm{i} u+\frac{1}{2\pi \mathrm{i} }\frac{T}{z}\right)\cdot F,
\end{eqnarray}
where $\mathrm{i}=\sqrt{-1}$, $h$ is a nonzero real parameter and $u=\sum_{i=1}^n u_i\otimes E_{ii}\in U\left({\mathfrak {gl}}_n\right) \otimes {\rm End}\left(\mathbb{C}^n\right)$ with $u_1,...,u_n$ complex parameters, $E_{ii}$ the $n\times n$ matrix whose $\left(i,i\right)$-entry is $1$ and other entries are $0$. Here the action of the coefficient matrix on $F\left(z\right)$ is given by matrix multiplication and the representation of ${\mathfrak {gl}}_n$ on $L(\lambda)$. The equation \eqref{chqeq} has a unique formal solution $\hat{F}(z)$ around
$z = \infty$. The standard theory of resummation states that there exist certain  sectorial regions around $z=\infty$, such that on each of these sectorial regions there is a unique (therefore canonical) holomorphic solution with the prescribed asymptotics $\hat{F}(z)$. These solutions are in general different (that reflects the Stokes phenomenon), and the transition between them can be measured by a pair of Stokes matrices $S_{h\pm}\left(u,T\right)\in{\rm End}\left(L\left(\lambda\right)\right)\otimes{\rm End}\left(\mathbb{C}^n\right)$. See \cite{Xu2} for more details.

The equation can be seen as a Knizhnik–Zamolodchikov (KZ) equation with irregular singularities \cite{FMTV, Re}. It was proved in \cite{Xu2} that the Stokes matrices of the equation \eqref{chqeq} give rise to a representation of quantum group $U_q(\frak{gl}_n)$ (with $q=\mathrm{e}^{\frac{h}{2}}$ and $h$ is a real number) on the vector space $L(\lambda)$. Furthermore, the WKB approximation (i.e., the $h\rightarrow \infty$ leading asymptotics) of the Stokes matrices was used in \cite{Xu3} to give the first transcendental realization of the crystal structures in the representation. See \cite{HKbook} and \cite{FRT} for the theory of quantum groups and crystal basis. Along the way, the special case of \eqref{chqeq}, as $u_1=u_2=\cdots=u_{n-1}=0$ and $u_n=1$, plays a crucial role. That is the equation
\begin{eqnarray}\label{introqeq}
\frac{\mathrm{d}F}{\mathrm{d}z}=h\left(\mathrm{i} E_n+\frac{1}{2\pi \mathrm{i} }\frac{T}{z}\right)\cdot F,
\end{eqnarray}
where $E_n=1\otimes E_{nn}\in U\left({\mathfrak {gl}}_n\right) \otimes {\rm End}\left(\mathbb{C}^n\right)$. The equation also plays an important role from the perspective of isomonodromy deformation \cite{JMU}.  The main result of this paper is the following theorem, which computes the Stokes matrices $S_{h\pm}\left(E_n,T\right)\in{\rm End}\left(L\left(\lambda\right)\right)\otimes{\rm End}\left(\mathbb{C}^n\right)$ of the system \eqref{introqeq} explicitly. See Section \ref{qsecextension}.

\begin{thm}\label{introqstokes}
The Stokes matrix $S_{h +}\left(E_{n},T\right)$ of the system \eqref{introqeq} takes the block matrix form
\begin{equation*}
S_{h+}\left(E_{n},T\right)=\left(\begin{array}{cc}
    \mathrm{e}^{\frac{-h T^{\left(n-1\right)}}{2}} & b_{h+}  \\
    0 & \mathrm{e}^{\frac{-he_{nn}}{2}}
  \end{array}\right),
  \end{equation*}
where $b_{h+}=\left(\left(b_{h+}\right)_1,...,\left(b_{h+}\right)_{n-1}\right)^{\intercal}$ with entries as follows
\begin{equation*}
\left(b_{h+}\right)_k= -\sum_{i=1}^{n-1}\sum_{j=1}^{n-1} \frac{(h\mathrm{e}^{-\frac{\pi \mathrm{i}}{2}})^{\frac{h}{2\pi \mathrm{i} }(\ell^{(n-1)}_i+n-2)}}{(h\mathrm{e}^{\frac{\pi \mathrm{i}}{2}})^{\frac{h }{2\pi \mathrm{i} }\left(e_{n n}+1\right)}} \frac{\prod_{l=1,l\neq i}^{n-1} \Gamma\left(1+\frac{h}{2\pi \mathrm{i} }(\ell^{(n-1)}_i-1-\ell^{(n-1)}_l)\right)}{\prod_{l=1}^{n}\Gamma\left(1+\frac{h}{2\pi \mathrm{i} }(\ell^{(n-1)}_i-1-\ell^{(n)}_l)\right)} (\mathrm{Pr}_{i})_{k j} \cdot h e_{j n}.
\end{equation*}
Here the elements $\ell^{\left(n-1\right)}_i\in \mathrm{End}\left(L\left(\lambda\right)\right)$, $\mathrm{Pr}_{i}\in {\rm End}\left(L\left(\lambda\right)\right)\otimes{\rm End}\left(\mathbb{C}^n\right)$, arising from representation theory, are given in Definition \ref{ell} and Definition \ref{def Pr}. And the use of the matrix gamma function is explained in \eqref{matrixgamma}.
\end{thm}
Similarly, one can compute explicitly the quantum Stokes matrix $S_{h-}(E_n,T)$. The proof of Theorem \ref{introqstokes} mainly relies on the known asymptotic behavior of confluent hypergeometric functions, which is also used in computation of Stokes matrices of the following special (classical) confluent hypergeometric system for a function $F\left(z\right)\in \mathrm{GL}\left(n,\mathbb{C}\right)$
\begin{equation} \label{intro:classical original equation} \frac{\mathrm{d}}{\mathrm{d} z} F\left(z\right)=\left(\mathrm{i} E_n-\frac{1}{2\pi \mathrm{i} }\frac{A}{z}\right) F\left(z\right), 
\end{equation}
where $A=\left(a_{i j}\right)\in \mathfrak{gl}_n$ and $\mathrm{GL}\left(n,\mathbb{C}\right)$ is the general linear group of degree $n$ over the field $\mathbb{C}$. More specifically, we use a gauge transformation that diagonalizes the upper left $(n-1)\times (n-1)$ submatrix of $A$. The Stokes matrices between the equation \eqref{intro:classical original equation} and the transformed equation 
\begin{equation}\label{diagsys}
    \frac{\mathrm{d}}{\mathrm{d} z} F\left(z\right)=\left(\mathrm{i} E_n-\frac{1}{2\pi \mathrm{i} }\frac{A_{n-1}}{z}\right) F\left(z\right),
\end{equation}
are conjugate (see Proposition \ref{10}), where $A_{n-1}$ is given in \eqref{zhuazi}. The solution $F\left(z;E_n,A_{n-1}\right)$ around $z=0$ of \eqref{diagsys} can be written in terms of confluent hypergeometric functions (see Proposition \ref{classical solution}). The transition between $F\left(z;E_n,A_{n-1}\right)$ and the unique solution $F_i\left(z;E_n,A_{n-1}\right)$ with the prescribed asymptotics on the sectorial region around $z=\infty$ is measured by some matrix $U_i$. Considering the asymptotic behavior of $F\left(z;E_n,A_{n-1}\right)$ on this sectorial region, we can write out $U_i$ explicitly (see Proposition \ref{mainpro}). Moreover, we obtain explicit expression of the Stokes matrices as quotients of two $U_i$ on adjacent sectorial regions (see Proposition \ref{S+ diag sys} and Proposition \ref{main stokes} for equations \eqref{intro:classical original equation} and \eqref{diagsys}). See \cite{BJL} for an example where $A$ is a $2\times 2$ matrix.

Besides asymptotics, we need to deal with commutation relations in $U\left(\mathfrak{gl}_n\right)$ when proving Theorem \ref{introqstokes}. Replacing minors, eigenvalues of $A$ by quantum minors (see Definition \ref{qminor}), quantum eigenvalues (see Definition \ref{ell}) of $T$, we can transform the equation \eqref{introqeq} into the following form 
\begin{equation}\label{q diag sys}
    \frac{\mathrm{d}F}{\mathrm{d}z}=h\left(\mathrm{i} E_n+\frac{1}{2\pi \mathrm{i} }\frac{T_{n-1}}{z}\right)\cdot F, 
\end{equation}
where $T_{n-1}$ is the diagonalization of the upper left $(n-1)\times (n-1)$ submatrix of $T$, given explicitly in \eqref{Tn-1}. Then we can write out the solution $F_h\left(z;E_n,T_{n-1}\right)$ (see Proposition \ref{qsolthm}) and the transition matrix $U_{h,i}$ (see Proposition \ref{6}). At last, we get the Stokes matrices in Theorem \ref{introqstokes} explicitly as in the classical case. Along the way, we need to use various commutation relations of quantum minors.

The formal solutions, actual solutions and Stokes matrices of equations \eqref{introqeq} and \eqref{intro:classical original equation} are related to each other by deformation quantization. More precisely, on the one hand, following Proposition \ref{uniformal}, the equation \eqref{introqeq} has a unique formal solution
\begin{equation*}
\hat{F}_h\left(z;E_n,T\right)=\hat{H}\left(z\right) \mathrm{e}^{{h\mathrm{i} E_nz}}z^{\frac{h\delta_{n-1}\left(T\right)}{2\pi\mathrm{i}}}, \ \ \ {\it for} \ \hat{H}=\mathrm{Id}+H_1z^{-1}+H_2z^{-2}+\cdot\cdot\cdot,
\end{equation*}
where $\delta_{n-1}\left(T\right)$ is defined in \eqref{qdeltan-1} and the coefficients $H_m\in {\rm End}(L(\lambda))\otimes{\rm End}(\mathbb{C}^n)$ satisfy the following recursive relation
\begin{eqnarray*}
\left[H_{m+1}, \mathrm{i}E_n\right]=\left(\frac{ m}{h}+\frac{T}{2\pi \mathrm{i}}\right)\cdot H_{m}-H_{m}\cdot   \frac{\delta_{n-1}\left(T\right)}{2\pi \mathrm{i}}.
\end{eqnarray*}
As a result, the $(k,l)$-entry $H_{m,kl}\in \mathrm{End}\left(L\left(\lambda\right)\right)$ of $H_m$ is given by (the representation on $L(\lambda)$ of) an element of filtration degree $2m$ in $U\left({\mathfrak {gl}}_n\right)$. 

On the other hand, the equation \eqref{intro:classical original equation} has a unique formal solution 
\begin{equation*}
\hat{F}\left(z;E_n,A\right)=\hat{h}(z)z^{-\frac{\delta_{n-1}\left(A\right)}{2\pi \mathrm{i} }}\mathrm{e}^{\mathrm{i} E_n z}, \ \ \ {\it for} \ \hat{h}=\mathrm{Id}+h_1z^{-1}+h_2z^{-2}+\cdot\cdot\cdot,
\end{equation*}
where $\delta_{n-1}(A)$ is defined in \eqref{deltan-1} and the coefficients $h_m\in {\rm End}(\mathbb{C}^n)$ satisfy the recursive relation
\begin{equation*}
    \left[h_{m+1},\mathrm{i} E_n\right]=\left(m-\frac{A}{2\pi \mathrm{i} }\right)\cdot h_m+h_m\cdot \frac{\delta_{n-1}\left(A\right)}{2\pi \mathrm{i} }.
\end{equation*}
We find that $h_m$ satisfies the same recursive relation as $H_m$ (provided replacing $h$,
$T$ and $\delta_{n-1} (T)$ by $1$, $-A$ and $-\delta_{n-1}(A)$) and the $(k,l)$-entry $h_{m,kl}$ of $h_m$ is a degree $2m$ polynomial of the variables $a_{i j}$. Thus, the equation \eqref{introqeq} provides us a canonical lift of the elements $h_{m,k l}$ from the Poisson algebra $\mathrm{Sym}\left(\mathfrak{gl}_n\right)$ to the elements $H_{m,k l}$ in the algebra $U\left(\mathfrak{gl}_n\right)$. 
 
The coefficients of the formal power series solution already encode all the analytic data. For example, the Stokes matrices $S_{h\pm}(E_n,T)$ and $S_\pm(E_n,A)$ of \eqref{introqeq} and \eqref{intro:classical original equation} are defined from the Borel-Laplace transforms of the power series $\sum_m H_m z^{-m}$ and $\sum_m h_mz^{-m}$ respectively. More directly, using the associated function method, see e.g., \cite[Chapter 9]{Balser} and \cite{LuS,Yokoyama1988}, the Stokes matrices can be obtained from the regularized limit of the matrices $h_m\in { \mathrm{Sym}}(\frak{gl}_n)\otimes {\rm End}(\mathbb{C}^n)$ and $H_m\in U(\frak{gl}_n)\otimes {\rm End}(\mathbb{C}^n)$ as $m\rightarrow \infty$ respectively. Therefore (after taking
the completion of $\mathrm{Sym}({\frak {gl}}_n)$ and the representation $L(\lambda)$ of $U(\frak{gl}_n)$), the Stokes matrices $S_{h\pm}(E_n,T)$ of \eqref{introqeq}, seen as $n\times n$ matrices with entries valued in ${\rm End}(L(\lambda))$, are the quantization of the classical Stokes matrices $S_\pm(E_n,A)$ of \eqref{intro:classical original equation}, seen as in $\widehat{ \mathrm{Sym}}(\frak{gl}_n)\otimes {\rm End}(\mathbb{C}^n)$
 ($n\times n$ matrix valued analytic functions of the parameter $A\in {\frak {gl}}_n^*$). By this reason, we call $S_{h\pm}(E_n,T)$ the quantum Stokes matrices. We refer 
 readers to \cite{Xu2} for the quantum Stokes matrices of the universal quantum differential equation \eqref{chqeq} with a second order pole, and to \cite{Xu3} for the quantum Stokes matrices of the universal quantum differential equation \eqref{chqeq} with an arbitrary order pole.



\vspace{3mm}
The organization of the paper is as follows. Section \ref{sec2} computes explicitly the Stokes matrices of the classical linear system \eqref{intro:classical original equation}. Section \ref{beginsec} computes explicitly the quantum Stokes matrices of the quantum linear system \eqref{introqeq}. The structure of Section \ref{beginsec} is parallel to the one of Section \ref{sec2}, that is, the subsections of Section \ref{sec2} have quantum analogs in Section \ref{beginsec}.
 
\subsection*{Acknowledgements}
\noindent
We would like to thank Yiming Ma and Qian Tang for their discussion and valuable feedback on the paper. The authors are supported by the National Key Research and Development Program of China (No. 2021YFA1002000) and by the National Natural Science Foundation of China (No. 12171006).

\section{Stokes matrices of a special (classical) confluent hypergeometric system}\label{sec2}
This section computes explicitly the Stokes matrices of the linear system \eqref{intro:classical original equation}. In particular, Section \ref{secStokes} gives the preliminaries of Stokes data of meromorphic linear systems \eqref{intro:classical original equation}. Section \ref{diagonalization}
introduces a gauge transformation that diagonalizes the upper left block of the equation \eqref{intro:classical original equation}. Section \ref{sol and stokes} studies asymptotics of the solution of the equation after gauge transformation and then derives the explicit expression of its Stokes matrices via the asymptotics. Section \ref{secextension} then obtains the Stokes matrices of the linear system \eqref{intro:classical original equation}. 
\subsection{Stokes matrices}\label{secStokes}
We first consider a special (classical) confluent hypergeometric system for a function $F\left(z\right)\in \mathrm{GL}\left(n,\mathbb{C}\right)$
\begin{equation} \label{classical original equation} \frac{\mathrm{d}}{\mathrm{d} z} F\left(z\right)=\left(\mathrm{i} E_n-\frac{1}{2\pi \mathrm{i} }\frac{A}{z}\right) F\left(z\right). 
\end{equation}
Let us denote by $A^{\left(k\right)}$ the upper left $k\times k$ submatrix of $A=\left(a_{i j}\right)\in \mathfrak{gl}_n$, and by $\lambda^{\left(k\right)}_1\left(A\right),...,\lambda^{\left(k\right)}_{k}\left(A\right)$, the eigenvalues of $A^{\left(k\right)}$. Let us make the assumption that 
\begin{equation}\label{nonres}
\frac{1}{2\pi \mathrm{i} }\left(\lambda^{\left(n\right)}_i\left(A\right)-\lambda^{\left(n\right)}_j\left(A\right)\right)\notin \mathbb{Z}_{>0}, \text{ for all }  1\leq i,j \leq n.
\end{equation}
Under the assumption, the system is nonresonant. Then

\begin{proposition}
The system \eqref{classical original equation} has a unique formal solution of the form
\begin{equation*}
\hat{F}\left(z;E_n,A\right)=\hat{h}(z)z^{-\frac{\delta_{n-1}\left(A\right)}{2\pi \mathrm{i} }}\mathrm{e}^{\mathrm{i} E_n z}, \ \ \ {\it for} \ \hat{h}=\mathrm{Id}+h_1z^{-1}+h_2z^{-2}+\cdot\cdot\cdot,
\end{equation*}   
where $h_m$ is a constant $n\times n$ matrix and 
\begin{equation}\label{deltan-1}
    \delta_{n-1}\left(A\right)_{ij}=\left\{
          \begin{array}{lr}
             a_{ij},   & \text{if} \ \ 1\le i, j\le n-1, \ \text{or} \ i=j=n;  \\
           0, & \text{otherwise}.
             \end{array}
\right.
\end{equation}
\end{proposition}

The radius of convergence of the formal power series $\sum_{m}h_mz^{-m}$ is in general zero. However, it follows from the general principal of differential equations with irregular singularities that (see e.g., \cite{Balser0,LR, MR}) the Borel resummation (Borel-Laplace transform) of $\sum_{m}h_mz^{-m}$ gives a holomorphic function in each Stokes sector (see Definition \ref{defSect}) around $z=\infty$. In this way, one gets an actual solution of \eqref{classical original equation} on each Stokes sector, see Proposition \ref{classical resummation}. These sectors are determined by the irregular term $\mathrm{i} E_n$ of the differential equation as follows.

Let $\widetilde{\mathbb{C}^*}$ be the Riemann surface of the natural logarithm. A point on $\widetilde{\mathbb{C}^*}$ is uniquely characterized by a pair $\left(r,\theta\right)$. Here $r>0$, $\theta\in \mathbb{R}$. Every point on $\widetilde{\mathbb{C}^*}$, or simply every pair $\left(r,\theta\right)$, corresponds to a complex number $z = r \mathrm{e}^{\mathrm{i}\theta}$, called the projection of $\left(r,\theta\right)$. We shall use the same symbol $z$ both for the point on $\widetilde{\mathbb{C}^*}$ and its projection in the punctured plane. We call $\theta$ the argument of the point $z$ on $\widetilde{\mathbb{C}^*}$, i.e., $\theta=\mathrm{arg}\left(z\right)$, and emphasize that on $\widetilde{\mathbb{C}^*}$ the argument of a point is uniquely defined. 

\begin{definition}
For every integer $i$, set $d_i=i\pi-\frac{\pi}{2}$. The anti-Stokes ray of \eqref{classical original equation} corresponding to $d_i$, is the set of all points on $\widetilde{\mathbb{C}^*}$ with argument $d_i$. The anti-Stokes ray is also denoted as $d_i$ by abuse of notation.
\end{definition}

\begin{definition}\label{defSect}
Let $S\left(a,b\right)=\{z\in \widetilde{\mathbb{C}^*}:\operatorname{arg}\left(z\right)\in \left(a,b\right)\}$, where $a$ and $b$ are real numbers. For every integer $i$, the Stokes sector of \eqref{classical original equation} is
    \begin{align*}
        \operatorname{Sect}_i=S\left(d_{i}-\frac{\pi}{2},d_{i+1}+\frac{\pi}{2}\right).
    \end{align*}
\end{definition}
Then the Borel resummation of the formal solution $\hat{F}\left(z;E_n,A\right)$ along the directions in $S\left(d_i, d_{i+1}\right)$ produces an actual solution $F_i\left(z;E_n,A\right)$ whose asymptotic expansion in $\operatorname{Sect}_i$ is $\hat{F}\left(z;E_n,A\right)$. That is
\begin{proposition}\label{classical resummation}
For $i\in \mathbb{Z}$, there exists a unique analytic solution $F_i\left(z;E_n,A\right)$ with the asymptotic expansion
\begin{equation*}
    F_i\left(z;E_n,A\right)\cdot z^{\frac{\delta_{n-1}\left(A\right)}{2\pi \mathrm{i} }}\mathrm{e}^{-\mathrm{i} E_n z}\sim \hat{h}\left(z\right),\quad z \rightarrow \infty \text{ in $\operatorname{Sect}_i$}.
\end{equation*}
\end{proposition}
In general, the solutions $F_i\left(z;E_n,A\right)$ corresponding to different sectors $\operatorname{Sect}_i$ are different.  

\begin{definition}\label{classical Stokes matrix}
    The Stokes matrices $S_{ \pm}\left(E_n, A\right)$ of the system \eqref{classical original equation} are the unique matrices such that
    \begin{itemize}
        \item on the sector $z\in S\left(d_0-\frac{\pi}{2},d_0+\frac{\pi}{2}\right)$, we have $F_{-1}\left(z;E_n,A\right)\cdot \mathrm{e}^{-\frac{\delta_{n-1}\left(A\right)}{2}} S_{+}\left(E_n, A\right) =F_{0}\left(z;E_n,A\right)$, and
        \item on the sector $z\in S\left(d_{1}-\frac{\pi}{2},d_{1}+\frac{\pi}{2}\right)$, we have $F_{0}\left(z;E_n,A\right)\cdot S_{-}\left(E_n, A\right)\mathrm{e}^{\frac{\delta_{n-1}\left(A\right)}{2}}=F_{-1}\left(z\mathrm{e}^{-2 \pi \mathrm{i} };E_n,A\right)$.
    \end{itemize}
\end{definition}

\begin{rmk}
We can define the Stokes matrix $S_i\left(E_n,A\right)$ for any integer $i$ by
\[F_{i-1}\left(z;E_n,A\right)\cdot \mathrm{e}^{-\frac{\delta_{n-1}\left(A\right)}{2}} S_{i}\left(E_n, A\right)  =F_i\left(z;E_n,A\right) \text{ on } S\left(d_i-\frac{\pi}{2},d_i+\frac{\pi}{2}\right).\]
However, one checks that 
\[S_i\left(E_n,A\right)=\mathrm {e}^{-\delta_{n-1}\left(A\right)}S_{i+2}\left(E_n,A\right)\mathrm{e}^{\delta_{n-1}\left(A\right)}.\]
Therefore, it is enough to study the pair of Stokes matrices $S_+\left(E_n,A\right)=S_0\left(E_n,A\right)$, and 
\begin{equation*}
    S_-\left(E_n,A\right)= \mathrm{e}^{-\frac{\delta_{n-1}\left(A\right)}{2}}S_{1}\left(E_n,A\right) \mathrm{e}^{\frac{\delta_{n-1}\left(A\right)}{2}}.
\end{equation*}


\end{rmk}

\subsection{Diagonalization of the upper left submatrix of $A$}\label{diagonalization}
We first compute the Stokes matrices under the conditions 
\begin{align}\label{ineq1}
\lambda^{\left(n-1\right)}_i\left(A\right) & \ne \lambda^{\left(n-1\right)}_j\left(A\right), \text{ for } \ 1\leq i\neq j \leq n-1,\\ \label{ineq2}
\lambda^{\left(n-1\right)}_i\left(A\right) & \ne \lambda^{\left(n-2\right)}_j\left(A\right), \text{ for } \ 1\leq i \leq n-1, \ \ \ 1\leq j \leq n-2 .
\end{align}
Along the way, we see that by analytic continuation, the result can be generalized to contain all nonresonant $A$. 

Let us take two $\left(n-1\right)\times \left(n-1\right)$ matrices
\begin{align*}
\mathcal{P}_{n-1}=\left(\begin{array}{ccc}
\left(-1\right)^{1+1} \Delta_{\hat{1} ,..., n-1}^{1 ,..., n-2}\left(A-\lambda^{\left(n-1\right)}_1\cdot \mathrm{Id}\right) & \cdots & \left(-1\right)^{1+n-1} \Delta_{\hat{1} ,..., n-1}^{1 ,..., n-2}\left(A-\lambda^{\left(n-1\right)}_{n-1}\cdot \mathrm{Id}\right) \\
\vdots &  \ddots & \vdots \\
\left(-1\right)^{n-1+1} \Delta_{1 ,..., \widehat{n-1}}^{1 ,..., n-2}\left(A-\lambda^{\left(n-1\right)}_1\cdot \mathrm{Id}\right) & \cdots & \left(-1\right)^{n-1+n-1} \Delta_{1 ,..., \widehat{n-1}}^{1 ,..., n-2}\left(A-\lambda^{\left(n-1\right)}_{n-1}\cdot \mathrm{Id}\right)
\end{array}\right),
    \end{align*}    
and
\begin{align*}
 \mathcal{Q}_{n-1}=\left(\begin{array}{ccc}
\left(-1\right)^{1+1} \Delta_{1 ,..., n-2}^{\hat{1} ,..., n-1}\left(A-\lambda^{\left(n-1\right)}_1\cdot \mathrm{Id}\right) & \cdots & \left(-1\right)^{1+n-1} \Delta_{1 ,..., n-2}^{1 ,..., \widehat{n-1} }\left(A-\lambda^{\left(n-1\right)}_1\cdot \mathrm{Id}\right) \\
\vdots & \ddots & \vdots \\
\left(-1\right)^{n-1+1} \Delta_{1 ,..., n-2}^{\hat{1} ,..., n-1 }\left(A-\lambda^{\left(n-1\right)}_{n-1}\cdot \mathrm{Id}\right) & \cdots & \left(-1\right)^{n-1+n-1} \Delta_{1 ,..., n-2}^{1 ,..., \widehat{n-1}}\left(A-\lambda^{\left(n-1\right)}_{n-1}\cdot \mathrm{Id}\right)
\end{array}\right).
\end{align*}
Here $\Delta_{1,...,\hat{j},..., n-1}^{1,...,n-2}\left(A-\lambda^{\left(n-1\right)}_i\cdot \mathrm{Id}\right)$ and $\Delta^{1,...,\hat{j},..., n-1}_{1,...,n-2}\left(A-\lambda^{\left(n-1\right)}_i\cdot \mathrm{Id}\right)$ are the $\left(n-2\right)\times \left(n-2\right)$ minor of the matrix $A-\lambda^{\left(n-1\right)}_i\cdot \mathrm{Id}$ (the $\hat{j}$ means that the column or row index $j$ is omitted).
\begin{proposition}\label{PQ}
We have
\begin{align}\label{PAAP}
A^{\left(n-1\right)} \mathcal{P}_{n-1}=\mathcal{P}_{n-1} A^{\left(n-1\right)}_{n-1}, \ \text{ and } \ \mathcal{Q}_{n-1} A^{\left(n-1\right)}=A^{\left(n-1\right)}_{n-1}\mathcal{Q}_{n-1} 
   ,
   \end{align}
   where $A^{\left(n-1\right)}$ is the upper left $\left(n-1\right)\times \left(n-1\right)$ submatrix of $A$, and $A^{\left(n-1\right)}_{n-1}=\mathrm{diag}\left(\lambda^{\left(n-1\right)}_1,\ldots,\lambda^{\left(n-1\right)}_{n-1}\right).$ 
Furthermore, we have that 
\begin{equation}\label{PQD}
\mathcal{Q}_{n-1}\mathcal{P}_{n-1}=D_{n-1},
\end{equation}
where $D_{n-1}={\rm diag}\left(\left(D_{n-1}\right)_1,....,\left(D_{n-1}\right)_{n-1}\right)$ is a diagonal matrix with the diagonal elements 
\begin{equation*}
\left(D_{n-1}\right)_i=
\prod_{l=1,l\ne i}^{n-1}\left(\lambda^{\left(n-1\right)}_l-\lambda^{\left(n-1\right)}_i\right)\prod_{l=1}^{n-2}\left(\lambda^{\left(n-2\right)}_l-\lambda^{\left(n-1\right)}_i\right), \text{ for } i=1,...,n-1.
\end{equation*}
\end{proposition}
First, the identity \eqref{PAAP} follows from the Laplace expansion. The identity \eqref{PQD} follows from the following lemma.

\begin{lem}\label{1}
 Let $\chi_{n-1}\left(x\right)$ be the polynomial $\mathrm{det}\left(A^{\left(n-1\right)}-x\cdot \mathrm{Id}_{n-1}\right)$, with ${\rm Id}_{n-1}$ the rank $n-1$ identity matrix, and $\chi_{n-1}^{\prime}\left(x\right)$ be the derivative of $\chi_{n-1}\left(x\right)$ with respect to $x$. For $1\leq k \leq n-1$, we have
\begin{align}\label{chapoly}
\left( A^{\left(n-1\right)}-\lambda_{k}^{\left(n-1\right)}\cdot \mathrm{Id}_{n-1}\right)^{* 2}=-\chi_{n-1}^{\prime}\left( \lambda_{k}^{\left(n-1\right)} \right)\left(A^{\left(n-1\right)}-\lambda_{k}^{\left(n-1\right)}\cdot \mathrm{Id}_{n-1}  \right)^{*},
\end{align}
where $\left(A^{\left(n-1\right)}-\lambda_{k}^{\left(n-1\right)}\cdot \mathrm{Id}_{n-1}  \right)^{*}$ is the adjoint matrix of $A^{\left(n-1\right)}-\lambda_{k}^{\left(n-1\right)} \cdot {\rm Id}_{n-1} $, and $\left(A^{\left(n-1\right)}-\lambda_{k}^{\left(n-1\right)}\cdot \mathrm{Id}_{n-1} \right)^{* 2}$ is the square of $\left(A^{\left(n-1\right)}-\lambda_{k}^{\left(n-1\right)}\cdot \mathrm{Id}_{n-1} \right)^{*}$.
\end{lem}

\begin{proof}
It is enough to show \eqref{chapoly} under the assumptions \eqref{ineq1} and \eqref{ineq2}. Under the assumptions, the matrices $\mathcal{P}_{n-1}$ and $\mathcal{Q}_{n-1}$ are invertible. By identity \eqref{PAAP} and $\left(A^{\left(n-1\right)}-x\cdot \mathrm{Id}_{n-1}\right)^{*}=\chi_{n-1}\left(x\right) \left(A^{\left(n-1\right)}-x\cdot \mathrm{Id}_{n-1}\right)^{-1}$, we have
\begin{align*}
\mathcal{P}_{n-1}^{-1}\cdot \left(A^{\left(n-1\right)}-x\cdot \mathrm{Id}_{n-1}\right)^{*} \cdot \mathcal{P}_{n-1}&=\chi_{n-1}\left(x\right) \cdot \left(A^{\left(n-1\right)}_{n-1}-x\cdot \mathrm{Id}_{n-1}\right)^{-1} \\
&=\operatorname{diag}\left(\frac{\prod_{i=1}^{n}\left(\lambda^{\left(n-1\right)}_i  -x\right)}{\lambda^{\left(n-1\right)}_1-x  }, \cdots,\frac{\prod_{i=1}^{n}\left(\lambda^{\left(n-1\right)}_i -x \right)}{\lambda^{\left(n-1\right)}_{n-1}-x  }\right).
\end{align*}

\noindent Thus
\begin{align*}
\mathcal{P}_{n-1}^{-1}\cdot \left(  A^{\left(n-1\right)}-\lambda^{\left(n-1\right)}_k\cdot \mathrm{Id}_{n-1}\right)^{*}\cdot  \mathcal{P}_{n-1}=\operatorname{diag}\left(0,\cdots,\prod_{i=1,i\neq k}^{n-1}\left( \lambda_{i}^{\left(n-1\right)}-\lambda_{k}^{\left(n-1\right)}   \right),\cdots,0\right).
\end{align*}
Since 
the diagonal matrix 
$    \mathcal{P}_{n-1}^{-1}\left(  A^{\left(n-1\right)}-\lambda^{\left(n-1\right)}_k\cdot \mathrm{Id}_{n-1}\right)^{*} \mathcal{P}_{n-1}+\chi_{n-1}^{\prime}\left(\lambda^{\left(n-1\right)}_k  \right){\cdot \mathrm{Id}_{n-1}}
$
has zero $\left(k,k\right)$-entry, so
\begin{align*}
\mathcal{P}_{n-1}^{-1}\left( A^{\left(n-1\right)}-\lambda^{\left(n-1\right)}_k\cdot \mathrm{Id}_{n-1}\right)^{*} \mathcal{P}_{n-1}\cdot \left(\mathcal{P}_{n-1}^{-1}\left(  A^{\left(n-1\right)}-\lambda^{\left(n-1\right)}_k\cdot \mathrm{Id}_{n-1}\right)^{*} \mathcal{P}_{n-1}+\chi_{n-1}^{\prime}\left(\lambda^{\left(n-1\right)}_k  \right){\cdot\mathrm{Id}_{n-1}}\right)=0.
\end{align*}
That is just $
\left(A^{\left(n-1\right)}-\lambda_{k}^{\left(n-1\right)} \cdot \mathrm{Id}_{n-1}  \right)^{* 2}=-\chi_{n-1}^{\prime}\left(\lambda_{k}^{\left(n-1\right)}  \right)\left(A^{\left(n-1\right)}-\lambda_{k}^{\left(n-1\right)}\cdot \mathrm{Id}_{n-1}   \right)^{*}.$
\end{proof}

\begin{proof}[Proof of Proposition \ref{PQ}]
    Following \eqref{PAAP}, we have
\[A^{\left(n-1\right)}_{n-1} \mathcal{Q}_{n-1} \mathcal{P}_{n-1}=\mathcal{Q}_{n-1}A^{\left(n-1\right)} \mathcal{P}_{n-1}=\mathcal{Q}_{n-1}\mathcal{P}_{n-1} A^{\left(n-1\right)}_{n-1},\]
which implies that $A^{\left(n-1\right)}_{n-1}$ commutes with $\mathcal{Q}_{n-1} \mathcal{P}_{n-1}$. 
Thus if $\lambda^{\left(n-1\right)}_k$, $1\leq k \leq n-1$, are distinct, we know $\mathcal{Q}_{n-1} \mathcal{P}_{n-1}$ is diagonal. Note that the $k$-th diagonal element of $\mathcal{Q}_{n-1} \mathcal{P}_{n-1}$ is just the $\left(n-1,n-1\right)$-entry of the left side of the matrix identity \eqref{chapoly} in Lemma \ref{1},
 it therefore proves the proposition for generic $A^{(n-1)}$. The result still holds (when eigenvalues of $A^{\left(n-1\right)}$ are not distinct) by analytic continuation.
\end{proof}

Therefore, under the conditions \eqref{ineq1} and \eqref{ineq2}, the matrices \[P_{n-1}:=\mathcal{P}_{n-1}\frac{1}{\sqrt{D_{n-1}}} \ \text{  and  } \    Q_{n-1}:=\frac{1}{\sqrt{D_{n-1}}}\mathcal{Q}_{n-1}\] are invertible to each other, so we can diagonalize the upper left $\left(n-1\right)\times \left(n-1\right)$ block to simplify the system \ref{classical original equation}.

In particular, let us use ${P}_{n-1}$ to diagonalize the upper left submatrix of $A$ and introduce
\begin{eqnarray}\label{zhuazi}
    A_{n-1}:=\operatorname{diag}\left({P}_{n-1}^{-1},1\right)\cdot A \cdot \operatorname{diag}\left({P}_{n-1},1\right) = \left(\begin{array}{cccc}
\lambda^{\left(n-1\right)}_1 & & 0 & a^{\left(n-1\right)}_1 \\
 & \ddots & & \vdots \\
0 & & \lambda^{\left(n-1\right)}_{n-1} & a^{\left(n-1\right)}_{n-1} \\
b^{\left(n-1\right)}_1 & \cdots & b^{\left(n-1\right)}_{n-1} & a_{n n} 
\end{array} \right),
\end{eqnarray}
and by definition the functions $a^{\left(n-1\right)}_i\left(A\right)$ and $b^{\left(n-1\right)}_i\left(A\right)$ of $A$ are
\begin{align*}
    a^{\left(n-1\right)}_i\left(A\right)&= \sum_{v=1}^{n-1}\left({P}_{n-1}^{-1}\right)_{i v} \cdot \left(A\right)_{v n}=\frac{(-1)^{i+n-1}\Delta^{1,...,n-1}_{1,...,n-2,n}\left(A-\lambda^{\left(n-1\right)}_i\right)}{\sqrt{\prod_{l=1,l\ne i}^{n-1}\left(\lambda^{\left(n-1\right)}_l-\lambda^{\left(n-1\right)}_i\right)\prod_{l=1}^{n-2}\left(\lambda^{\left(n-2\right)}_l-\lambda^{\left(n-1\right)}_i\right)}},\\ 
    b^{\left(n-1\right)}_i\left(A\right)&= \sum_{v=1}^{n-1} \left(A\right)_{n v}\left({P}_{n-1}\right)_{v i}=\frac{(-1)^{i+n-1}\Delta_{1,...,n-1}^{1,..., n-2,n}\left(A-\lambda_i^{\left(n-1\right)}\right)}{\sqrt{\prod_{l=1,l\ne i}^{n-1}\left(\lambda^{\left(n-1\right)}_l-\lambda^{\left(n-1\right)}_i\right)\prod_{l=1}^{n-2}\left(\lambda^{\left(n-2\right)}_l-\lambda^{\left(n-1\right)}_i\right)}}.
\end{align*}

Taking the Laplace expansion of the determinant $\operatorname{det}\left(x+\frac{1}{2\pi\mathrm{i} }A_{n-1}\right)$ of the matrix $A_{n-1}$ in \eqref{zhuazi} with respect to the last column (or row), we get
\begin{align}\label{x identity}
    \frac{\prod_{i=1}^{n}\left(x+\frac{\lambda^{\left(n\right)}_i}{2\pi\mathrm{i} }\right)}{\prod_{i=1}^{n-1}\left(x+\frac{\lambda^{\left(n-1\right)}_i}{2\pi\mathrm{i} }\right)}=x+\frac{a_{n n}}{2\pi\mathrm{i} }-\sum_{k=1}^{n-1}\frac{b^{\left(n-1\right)}_k}{2\pi\mathrm{i} } \frac{1}{x+\frac{\lambda^{\left(n-1\right)}_k}{2\pi\mathrm{i} }} \frac{a^{\left(n-1\right)}_k}{2\pi\mathrm{i} }.
\end{align}
In the above identity, replacing $x$ by $m-\frac{1}{2\pi \mathrm{i} }\lambda^{\left(n\right)}_j$, $m\in \mathbb{Z}_{>0}$, leads to following lemma that will be used later
\begin{lem}\label{ciq}
Under the conditions \eqref{ineq1} and \eqref{ineq2}, we have 
\begin{align*}
    \frac{\prod_{i=1}^n\left(m-\frac{1}{2\pi \mathrm{i} }\lambda^{\left(n\right)}_j+\frac{1}{2\pi \mathrm{i} }\lambda^{\left(n\right)}_i\right)}{\prod_{i=1}^{n-1}\left(m-\frac{1}{2\pi \mathrm{i} }\lambda^{\left(n\right)}_j+\frac{1}{2\pi \mathrm{i} }\lambda^{\left(n-1\right)}_i\right)}=m-\frac{1}{2\pi \mathrm{i} }\lambda^{\left(n\right)}_j+\frac{1}{2\pi \mathrm{i} }a_{n n}-\sum_{k=1}^{n-1}\frac{b^{\left(n-1\right)}_k}{2\pi \mathrm{i} } \frac{1}{m-\frac{1}{2\pi \mathrm{i} }\lambda^{\left(n\right)}_j+\frac{1}{2\pi \mathrm{i} }\lambda^{\left(n-1\right)}_k}\frac{a^{\left(n-1\right)}_k}{2\pi \mathrm{i} }.
\end{align*}
\end{lem}

Now let us use the conjugation of the $n\times n$ matrix $\operatorname{diag}\left({P}_{n-1},1\right)$ to diagonalize the upper left part of the equation \eqref{classical original equation}.
Since the matrices $E_n$ and $\operatorname{diag}\left({P}_{n-1},1\right)$ commute with each other, we get
\begin{proposition}\label{10}
Under the conditions \eqref{ineq1} and \eqref{ineq2}, if $F\left(z;E_n,A_{n-1}\right)$ is a fundamental solution of 
\begin{align}\label{classical new equation}
    \frac{\mathrm{d}}{\mathrm{d} z} F\left(z\right)=\left(\mathrm{i} E_n-\frac{1}{2\pi \mathrm{i} }A_{n-1}z^{-1}\right) F\left(z\right),
\end{align}
then $\operatorname{diag}\left({P}_{n-1},1\right) F\left(z;E_n,A_{n-1}\right)\operatorname{diag}\left({P}_{n-1}^{-1},1\right)$ is a fundamental solution of the equation \eqref{classical original equation}, where $A_{n-1}$ is given in \eqref{zhuazi}. Moreover, the Stokes matrices of the systems \eqref{classical original equation} and \eqref{classical new equation} are related by
\begin{equation*}
S_{\pm}\left(E_n, A\right)=\operatorname{diag}\left({P}_{n-1},1\right)S_{\pm}\left(E_n, A_{n-1}\right)\operatorname{diag}\left({P}_{n-1}^{-1},1\right).
\end{equation*}
\end{proposition}
\subsection{Solutions of \eqref{classical new equation} and its Stokes matrices}\label{sol and stokes}
In this subsection, we compute the explicit expression of the Stokes matrices $S_\pm\left(E_n,A_{n-1}\right)$. The expression was already computed via different methods, like the associated function method, topological method and so on. See e.g. \cite{Balser, DHMS, DM, MH}.   Here we do the computation via the most direct way, i.e., via the known asymptotics of confluent hypergeometric functions. This approach can be directly applied to the quantum case, see Section {\ref{beginsec}}.

\begin{definition}\label{ghf}
A generalized hypergeometric function ${}_{p}F_{q}\left(\alpha_1,...,\alpha_p;\beta_1,...,\beta_q;z\right)$ associated to any $\alpha_i\in \mathbb{C}$, for $1\leq i \leq p$, and $\beta_j\in \mathbb{C}\backslash\{0,-1,-2,\cdots\}$, for $1\leq j \leq q$, is
    \begin{align*}
        {}_{p}F_{q}\left(\alpha_1,...,\alpha_p;\beta_1,...,\beta_q;z\right):=\sum_{k=0}^{\infty}\frac{\prod_{i=1}^p\left(\alpha_i\right)_k}{\prod_{j=1}^q\left(\beta_j\right)_k }\frac{z^k}{k!},
    \end{align*}
where $\left(\alpha\right)_0=1$ and $\left(\alpha\right)_k=\alpha \left(\alpha+1\right) \cdots \left(\alpha+k-1\right)$, for $k\geq 1$.
When $p=q$, the radius of convergence of ${}_{p}F_{q}\left(\alpha_1,...,\alpha_p;\beta_1,...,\beta_q;z\right)$, which is called a confluent hypergeometric function, is $\infty$.
\end{definition}

\begin{proposition}\label{classical solution}
Under the nonresonant condition \eqref{nonres} and the conditions \eqref{ineq1}, \eqref{ineq2}, the system \eqref{classical new equation} has a fundamental solution 
\begin{equation}\label{diagFz}
    F\left(z;E_n,A_{n-1}\right)=Y \cdot H\left(z\right) \cdot z^{-\frac{1}{2\pi \mathrm{i} }A_{n}},
\end{equation}
where $Y=\operatorname{diag}\left(a_{1}^{\left(n-1\right)}, ..., a_{n-1}^{\left(n-1\right)}, 1\right)$, and $A_{n}=\operatorname{diag}\left(\lambda^{\left(n\right)}_1,...,\lambda^{\left(n\right)}_{n}\right)$. The matrix $H\left(z\right)$ is an $n\times n$ matrix with entries as follows, for $1 \leq i \leq n-1$ and $1 \leq j \leq n$,
\begin{align*}
	&H\left(z\right)_{i j}=\frac{1}{\lambda_{j}^{\left(n\right)}-\lambda_{i}^{\left(n-1\right)}} \cdot{ }_{n-1} F_{n-1}\left(a_{i j, 1}, ..., a_{i j, n-1}; b_{i j, 1}, ..., \widehat{b_{i j, j}}, ..., b_{i j, n} ;\mathrm{i} z\right), \\
	&H\left(z\right)_{n j}={ }_{n-1} F_{n-1}\left(a_{n j, 1}, ..., a_{n j, n-1}; b_{n j, 1}, ..., \widehat{b_{n j, j}}, ..., b_{n j, n} ;\mathrm{i}  z\right),
\end{align*}
where the variables $\left\{a_{i j, l}\right\}$ and $\left\{b_{i j, l}\right\}$ are as follows
\begin{align*}
&a_{i j, i} =\frac{1}{2 \pi \mathrm{i} }\left(\lambda_{i}^{\left(n-1\right)}-\lambda_{j}^{\left(n\right)}\right), \quad 1\leq i \leq n-1, \, 1\leq j \leq n,\\
&a_{i j, l}=1+\frac{1}{2 \pi \mathrm{i} }\left(\lambda_{l}^{\left(n-1\right)}-\lambda_{j}^{\left(n\right)}\right),\quad 1 \leq i, j \leq n,\, 1 \leq l \leq n-1, \, l \neq i, \\
&b_{i j, l}=1+\frac{1}{2 \pi \mathrm{i} }\left(\lambda_{l}^{\left(n\right)}-\lambda_{j}^{\left(n\right)}\right), \quad 1 \leq i, j,l \leq n .
\end{align*}
The notation $\widehat{b_{i j, j}}$ means the term $b_{i j, j}$ is skipped, for $1\leq i \leq n$.
\end{proposition}
\begin{proof}
The series expansion of the matrix function \eqref{diagFz} is \begin{equation}\label{seriesexp}
    F\left(z;E_n,A_{n-1}\right)=\sum_{m=0}^{\infty} L_{m} z^{m} z^{-\frac{1}{2 \pi \mathrm{i} } A_{n}},
\end{equation}
where $L_m$ is the coefficient of the expansion of $Y H\left(z\right)=\sum L_m z^m$. Following Definition \ref{ghf}, the 
$n\times n$ matrix $L_m$ has entries
        \begin{align}\label{LM1}
            &\left(L_m\right)_{i j}=a^{\left(n-1\right)}_i \frac{\mathrm{i} ^m}{m!}\frac{1}{\lambda^{\left(n\right)}_j-\lambda^{\left(n-1\right)}_i}\frac{\prod_{l=1}^{n-1}\left(a_{i j,l}\right)_m}{\prod_{l=1,l\ne j}^{n}\left(b_{ij,l}\right)_m}, \ \ 1\leq i \leq n-1, \ \  1\leq j \leq n,\\ \label{LM2}
            &\left(L_m\right)_{nj}=\frac{\mathrm{i} ^m}{m!}\frac{\prod_{l=1}^{n-1}\left(a_{n j,l}\right)_m}{\prod_{l=1,l\ne j}^{n}\left(b_{nj,l}\right)_m}.
        \end{align}
Let us verify that \eqref{seriesexp}
is indeed a series solution of the system \eqref{classical new equation}. Plugging \eqref{seriesexp}
into \eqref{classical new equation} and comparing the coefficients, we only need to verify the recursion relation
\begin{align}\label{L0 requirement}
L_{0} A_{n}&=A_{n-1} L_{0}, \\
\label{recur}\allowdisplaybreaks
    m L_{m}-\frac{1}{2\pi\mathrm{i} }L_{m}A_{n}&=\mathrm{i} E_n L_{m-1}-\frac{1}{2\pi \mathrm{i} } A_{n-1} L_{m}, \ \ m\ge 1.
\end{align}
The matrix identity \eqref{L0 requirement} can be verified directly via the explicit expression of
\begin{align*}
    L_0=\left(\begin{array}{ccc}
 \frac{a_{1}^{\left(n-1\right)}}{\lambda_{1}^{\left(n\right)}-\lambda_{1}^{\left(n-1\right)}} & \cdots &  \frac{a_{1}^{\left(n-1\right)}}{\lambda_{n}^{\left(n\right)}-\lambda_{1}^{\left(n-1\right)}} \\
\vdots & & \vdots \\
 \frac{a_{n-1}^{\left(n-1\right)}}{\lambda_{1}^{\left(n\right)}-\lambda_{n-1}^{\left(n-1\right)}} & \cdots & \frac{a_{n-1}^{\left(n-1\right)} }{\lambda_{n}^{\left(n\right)}-\lambda_{n-1}^{\left(n-1\right)}} \\
1 & \cdots & 1
\end{array}\right).
\end{align*}
On the one hand, the first $n-1$ rows of the matrix identity \eqref{recur} are just
\begin{align}\label{first n-1 row}
    \left(m-\frac{1}{2\pi \mathrm{i} }\lambda^{\left(n\right)}_j\right)\left(L_m\right)_{i j}=-\frac{1}{2\pi \mathrm{i} }\lambda^{\left(n-1\right)}_i \left(L_m\right)_{i j}-\frac{1}{2\pi \mathrm{i} }a^{\left(n-1\right)}_i \left(L_m\right)_{n j},\quad 1\leq i\leq n-1,\, 1\leq j \leq n.
\end{align}
Replacing $\left(L_m\right)_{ij}$ and $\left(L_m\right)_{nj}$ by \eqref{LM1} and \eqref{LM2}, one checks that \eqref{first n-1 row} is indeed an identity. The last row of the matrix identity \eqref{recur} is 
\begin{align}\label{last row}
    \left(m-\frac{1}{2\pi \mathrm{i} }\lambda^{\left(n\right)}_j\right)\left(L_m\right)_{n j}=\mathrm{i}  \left(L_{m-1}\right)_{n j}-\frac{1}{2\pi \mathrm{i} }\sum_{k=1}^{n-1}b^{\left(n-1\right)}_k \left(L_m\right)_{k j}-\frac{1}{2\pi \mathrm{i} } a_{n n}\left(L_m\right)_{n j}.
\end{align}
To verify \eqref{last row}, by
substituting \eqref{first n-1 row} into \eqref{last row}, we should prove
\begin{align}\label{substitute first n-1 rows into the last row}
    \left(m-\frac{1}{2\pi \mathrm{i} }\lambda^{\left(n\right)}_j-\sum_{k=1}^{n-1}\frac{b^{\left(n-1\right)}_k}{2\pi \mathrm{i} } \frac{1}{m-\frac{1}{2\pi \mathrm{i} }\lambda^{\left(n\right)}_j+\frac{1}{2\pi \mathrm{i} }\lambda^{\left(n-1\right)}_k}\frac{a^{\left(n-1\right)}_k}{2\pi \mathrm{i} }+\frac{1}{2\pi \mathrm{i} }a_{n n}\right) \left(L_m\right)_{n j}= \mathrm{i}  \left(L_{m-1}\right)_{n j},\quad 1\leq j \leq n.
\end{align}
Replacing the left hand side of \eqref{substitute first n-1 rows into the last row} by the left side of the identity in Lemma \ref{ciq}, we see that \eqref{substitute first n-1 rows into the last row} is equivalent to
\begin{align}\label{recursion insert x}
    \frac{\prod_{i=1}^n\left(m-\frac{1}{2\pi \mathrm{i} }\lambda^{\left(n\right)}_j+\frac{1}{2\pi \mathrm{i} }\lambda^{\left(n\right)}_i\right)}{\prod_{i=1}^{n-1}\left(m-\frac{1}{2\pi \mathrm{i} }\lambda^{\left(n\right)}_j+\frac{1}{2\pi \mathrm{i} }\lambda^{\left(n-1\right)}_i\right)}\left(L_m\right)_{n j}=\mathrm{i} \left(L_{m-1}\right)_{n j}.
\end{align}
In the end, the identity $\eqref{recursion insert x}$ is verified directly via the expression of $L_m$ in \eqref{LM2}.

\end{proof}

Now let $\hat{F}\left(z;E_n,A_{n-1}\right)=\left(\mathrm{Id}+\sum_{m=1}^{\infty}K_m z^{-m}\right)z^{-\frac{\delta_{n-1}\left(A_{n-1}\right)}{2\pi \mathrm{i} }}\mathrm{e}^{\mathrm{i} E_n z}$ be the unique formal solution of \eqref{classical new equation}. Here recall that by definition of the operator $\delta_{n-1}$ and the form of $A_{n-1}$, $\delta_{n-1}\left(A_{n-1}\right)$ is just the diagonal part of the matrix $A_{n-1}$.
Let $F_0\left(z;E_n,A_{n-1}\right)$, $F_{-1}\left(z;E_n,A_{n-1}\right)$ be the corresponding actual solutions in ${\rm Sect}_0$ and ${\rm Sect}_{-1}$ given as in Section \ref{secStokes}. Then 
    
\begin{proposition}\label{mainpro}
    Under the assumption of Proposition \ref{classical solution}, the $n\times n$ transition matrices $U_0$ and $U_{-1}$, defined by
\begin{align}\label{U0}
&F\left(z;E_n,A_{n-1}\right) =F_0\left(z;E_n,A_{n-1}\right) \cdot Y U_{0},\quad z\in \operatorname{Sect}_0,\\ \label{U1}
&F\left(z;E_n,A_{n-1}\right)=F_{-1}\left(z;E_n,A_{n-1}\right) \cdot Y U_{-1},\quad z\in \operatorname{Sect}_{-1},
\end{align}
have the expression as follows, for $1\leq i \leq n-1$, and $1\leq j \leq n$,
\begin{align*}
    &\left(U_0\right)_{i j}=\frac{1}{\lambda^{\left(n\right)}_j-\lambda^{\left(n-1\right)}_i}\frac{\prod_{l=1,l\ne j}^{n} \Gamma\left(b_{ij,l}\right)}{\prod_{l=1,l\ne i}^{n-1} \Gamma\left(a_{ij,l}\right)}\frac{\prod_{l=1,l\ne i}^{n-1} \Gamma\left(a_{ij,l}-a_{ij,i}\right)}{\prod_{l=1,l\ne j}^{n}\Gamma\left(b_{ij,l}-a_{ij,i}\right)}\frac{1}{\mathrm{e}^{-\frac{\pi \mathrm{i} }{2}a_{ij,i}}},\\
    &\left(U_{-1}\right)_{i j}=\frac{1}{\lambda^{\left(n\right)}_j-\lambda^{\left(n-1\right)}_i}\frac{\prod_{l=1,l\ne j}^{n} \Gamma\left(b_{ij,l}\right)}{\prod_{l=1,l\ne i}^{n-1} \Gamma\left(a_{ij,l}\right)}\frac{\prod_{l=1,l\ne i}^{n-1} \Gamma\left(a_{ij,l}-a_{ij,i}\right)}{\prod_{l=1,l\ne j}^{n}\Gamma\left(b_{ij,l}-a_{ij,i}\right)}\frac{1}{\mathrm{e}^{\frac{3\pi \mathrm{i} }{2}a_{ij,i}}},
\end{align*}
and for $1\leq j\leq n$,
\begin{align*}
\left(U_0\right)_{nj}=\left(U_{-1}\right)_{nj}=\frac{\prod_{l=1,l \ne j}^{n} \Gamma\left(b_{nj,l}\right)}{\prod_{l=1}^{n-1}\Gamma\left(a_{nj,l}\right)}\frac{1}{\mathrm{e}^{\frac{\pi \mathrm{i} }{2}a_{nj,n}}},
\end{align*}
where for convenience, we denote
\begin{align*}
&\lambda^{\left(n-1\right)}_n:=a_{n n}
=\sum_{l=1}^{n}\lambda^{\left(n\right)}_l-\sum_{l=1}^{n-1}\lambda^{\left(n-1\right)}_l,\\
    &a_{n j,n}:=\sum_{l=1,l\ne j}^{n}b_{nj,l}-\sum_{l=1}^{n-1}a_{nj,l}=\frac{1}{2\pi \mathrm{i} }\left(\lambda^{\left(n-1\right)}_n-\lambda^{\left(n\right)}_j\right),
    \quad  1\leq j \leq n.
\end{align*}
\end{proposition}
\begin{proof}
    Since $F\left(z;E_n,A_{n-1}\right)$ and $F_0\left(z;E_n,A_{n-1}\right)$ are both solutions of the equation \eqref{classical new equation}, the existence of $U_0$ is obvious. 
    From the asymptotics of $F_0\left(z;E_n,T_{n-1}\right)$ within ${\rm Sect}_0$ and the identity \eqref{U0}, we have
\begin{align}\label{F asym}
    F\left(z;E_n,A_{n-1}\right) U_{0}^{-1} Y^{-1} z^{\frac{\delta_{n-1}\left(A_{n-1}\right)}{2\pi \mathrm{i} }}\mathrm{e}^{-\mathrm{i} E_n z}\sim \operatorname{Id}+\sum_{m=1}^{\infty}K_m z^{-m},\quad z\rightarrow \infty \text{ within } \operatorname{Sect}_0.
\end{align}
Following the recursion of the coefficients of the formal fundamental solution, we have
\begin{align*}
    \left[K_{m+1},\mathrm{i} E_n\right]=-\frac{1}{2\pi \mathrm{i} }A_{n-1}K_m+K_m\left(\frac{1}{2\pi \mathrm{i} }\delta_{n-1}\left(A_{n-1}\right)+m\right).
\end{align*}
So we have for $1\le i\neq k \le n-1,$
\begin{align}\label{F1}
    &\left(K_1\right)_{i k}=-\mathrm{i} \frac{a^{\left(n-1\right)}_i}{a^{\left(n-1\right)}_k}\frac{\prod_{l=1}^{n}\frac{1}{2\pi \mathrm{i} }\left(\lambda^{\left(n\right)}_l-\lambda^{\left(n-1\right)}_k\right)}{\prod^n_{l=1,\, l\ne k}\frac{1}{2\pi \mathrm{i} }\left(\lambda^{\left(n-1\right)}_l-\lambda^{\left(n-1\right)}_k\right)\left(\frac{1}{2\pi \mathrm{i} }\left(\lambda^{\left(n-1\right)}_i-\lambda^{\left(n-1\right)}_k\right)-1\right)},\\
    \label{F2}
    &\left(K_1\right)_{n k}=-\frac{2\pi}{a^{\left(n-1\right)}_k}\frac{\prod^{n}_{l=1}\frac{1}{2\pi \mathrm{i} }\left(\lambda^{\left(n\right)}_l-\lambda^{\left(n-1\right)}_k\right)}{\prod^{n-1}_{l=1,l\ne k}\frac{1}{2\pi \mathrm{i} }\left(\lambda^{\left(n-1\right)}_l-\lambda^{\left(n-1\right)}_k\right)},\\
    \label{Fn}
    &\left(K_1\right)_{i n}=\frac{a^{\left(n-1\right)}_i}{2\pi}.
\end{align}
Here $\eqref{F1}$, $\eqref{F2}$ and $\eqref{Fn}$ are meaningful because if we consider the residue of \eqref{x identity} at $x=\frac{\lambda^{\left(n-1\right)}_i}{2\pi \mathrm{i}}$, we have
\begin{equation*}
    b^{\left(n-1\right)}_i a^{\left(n-1\right)}_i=-\frac{\prod_{l=1}^{n}\left(\lambda^{\left(n\right)}_l-\lambda^{\left(n-1\right)}_i\right)}{\prod_{l=1,l\neq i}^{n-1}\left(\lambda^{\left(n-1\right)}_l-\lambda^{\left(n-1\right)}_i\right)},
\end{equation*}
whose right side is nonzero under the assumption of Proposition \ref{classical solution}, which ensures that $a^{\left(n-1\right)}_i\neq 0$. Moreover, for $1\leq i\neq j \leq n$, we have
\begin{equation*}
    \left(K_1\right)_{i j}\neq 0,
\end{equation*}
implying that (together with \eqref{F asym}) for $z\rightarrow \infty$ within $\mathrm{Sect}_0$,
\begin{align}\label{F entry asym1}
    &\sum_{l=1}^n F\left(z;E_n,A_{n-1}\right)_{i l} \left(U_0^{-1}\right)_{l i} Y_{i i}^{-1} z^{\frac{\lambda^{(n-1)}_i}{2\pi \mathrm{i}}} \mathrm{e}^{-\mathrm{i} \delta_{i n} z}\sim 1+\sum_{m\geq 1}\left(K_{m}\right)_{i i}z^{-m},\quad 1\leq i \leq n,\\
    &\sum_{l=1}^n F\left(z;E_n,A_{n-1}\right)_{i l} \left(U_0^{-1}\right)_{l j} Y_{j j}^{-1} z^{\frac{\lambda^{(n-1)}_j}{2\pi \mathrm{i}}} \mathrm{e}^{-\mathrm{i} \delta_{j n} z}\sim \sum_{m\geq 1}\left(K_{m}\right)_{i j}z^{-m},\quad 1\leq i\neq j \leq n.\label{F entry asym2}
\end{align}
Here the symbol $f\left(z\right)\sim g\left(z\right)$ in \eqref{F entry asym1}, \eqref{F entry asym2}, for some functions $f\left(z\right)$, $g\left(z\right)$, implies that $\frac{f\left(z\right)}{g\left(z\right)}\rightarrow 1$. Moreover, we can multiply the exponential functions and power functions in \eqref{F entry asym1}, \eqref{F entry asym2} to the right side and discard higher order terms, so
\begin{align}\label{F aysm eright1}
    &\sum_{l=1}^n F\left(z;E_n,A_{n-1}\right)_{i l} \left(U_0^{-1}\right)_{l i}\sim Y_{i i} z^{-\frac{\lambda^{(n-1)}_i}{2\pi \mathrm{i}}} \mathrm{e}^{\mathrm{i} \delta_{i n} z},\quad 1\leq i \leq n,\\
    &\sum_{l=1}^n F\left(z;E_n,A_{n-1}\right)_{i l} \left(U_0^{-1}\right)_{l j}\sim \frac{\left(K_{1}\right)_{i j}}{z} Y_{j j} z^{-\frac{\lambda^{(n-1)}_j}{2\pi \mathrm{i}}} \mathrm{e}^{\mathrm{i} \delta_{j n} z},\quad 1\leq i\neq j \leq n,\label{F aysm eright2}
\end{align}
and \eqref{F aysm eright1}, \eqref{F aysm eright2} in turn imply, for $1\leq i \leq n-1$ and $1\leq j \leq n$, when $z\rightarrow \infty$ in $\operatorname{Sect}_0$,
\begin{equation}
\begin{aligned}\label{Fij asymptotic behavior}
    F\left(z;E_n,A_{n-1}\right)_{i j}\sim &\sum_{k=1,k\ne i}^{n-1}\left(K_1\right)_{i k}z^{-\frac{1}{2\pi \mathrm{i} }\lambda^{\left(n-1\right)}_k-1}a^{\left(n-1\right)}_k\left(U_0\right)_{k j}\\
    &+z^{-\frac{1}{2\pi \mathrm{i} }\lambda^{\left(n-1\right)}_i}a^{\left(n-1\right)}_i\left(U_0\right)_{i j}\\
    &+\left(K_1\right)_{in}z^{-\frac{1}{2\pi \mathrm{i} }\lambda^{\left(n-1\right)}_n-1}\operatorname{e}^{\mathrm{i} z} \left(U_0\right)_{nj}
\end{aligned}
\end{equation}
and
\begin{align}\label{Fnj asymptotic behavior}
    F\left(z;E_n,A_{n-1}\right)_{n j}\sim \sum_{k=1}^{n-1}\left(K_1\right)_{n k}z^{-\frac{1}{2\pi \mathrm{i} }\lambda^{\left(n-1\right)}_k-1}a^{\left(n-1\right)}_k\left(U_0\right)_{k j}+z^{-\frac{1}{2\pi \mathrm{i} }\lambda^{\left(n-1\right)}_n}\operatorname{e}^{\mathrm{i} z} \left(U_0\right)_{nj}.
\end{align}

To find the coefficients on the right hand sides of \eqref{Fij asymptotic behavior} and \eqref{Fnj asymptotic behavior} (therefore to get the entries of $U_0$), we only need to study the asymptotics of the solution $F\left(z;E_n,A_{n-1}\right)$ given in Proposition \ref{classical solution}. It is computed using the following known asymptotic behavior of confluent hypergeometric functions, see. e.g., \cite{OLBC, Wasow}. The following lemma is derived from the expansions for large variable of confluent hypergeometric functions in \cite[Chapter 16]{OLBC}.
\begin{lem}\label{nfn}
If $z\in S\left(-\frac{\pi}{2},\frac{3\pi}{2}\right)$, the following asymptotic behavior holds
\begin{equation}
\begin{aligned}\label{asy 0 pi}
    {}_{n}F_{n}\left(\alpha_1,\cdots,\alpha_n;\beta_1,\cdots,\beta_n;z\right)\sim&\frac{\prod_{l=1}^{n}\Gamma\left(\beta_l\right)}{\prod_{l=1}^{n}\Gamma\left(\alpha_l\right)}\sum_{m=1}^n\Gamma\left(\alpha_m\right)\frac{\prod_{l=1,l\neq m}^n\Gamma\left(\alpha_l-\alpha_m\right)}{\prod_{l=1}^n\Gamma\left(\beta_l-\alpha_m\right)}\frac{1}{\left(z \mathrm{e}^{-\pi \mathrm{i} }\right)^{{\alpha_m}}} \\&+\frac{\prod_{l=1}^{n}\Gamma\left(\beta_l\right)}{\prod_{l=1}^{n}\Gamma\left(\alpha_l\right)}\operatorname{e}^zz^{\sum_{l=1}^n\left(\alpha_l-\beta_l\right)}.
\end{aligned}
\end{equation}
Similarly, if $z\in S\left(-\frac{3\pi}{2},\frac{\pi}{2}\right)$, then 
\begin{equation}
\begin{aligned}\label{asy -pi 0}
    {}_{n}F_{n}\left(\alpha_1,\cdots,\alpha_n;\beta_1,\cdots,\beta_n;z\right)\sim&\frac{\prod_{l=1}^{n}\Gamma\left(\beta_l\right)}{\prod_{l=1}^{n}\Gamma\left(\alpha_l\right)}\sum_{m=1}^n\Gamma\left(\alpha_m\right)\frac{\prod_{l=1,l\neq m}^n\Gamma\left(\alpha_l-\alpha_m\right)}{\prod_{l=1}^n\Gamma\left(\beta_l-\alpha_m\right)}\frac{1}{\left(z \mathrm{e}^{\pi \mathrm{i} }\right)^{{\alpha_m}}} \\&+\frac{\prod_{l=1}^{n}\Gamma\left(\beta_l\right)}{\prod_{l=1}^{n}\Gamma\left(\alpha_l\right)}\operatorname{e}^zz^{\sum_{l=1}^n\left(\alpha_l-\beta_l\right)}.
\end{aligned}
\end{equation}
Here the symbol $f\left(z\right)\sim g\left(z\right)$ means the limit of $\frac{f\left(z\right)}{g\left(z\right)}$ is 1, when $z\rightarrow \infty$ in a required sector of $\widetilde{\mathbb{C}^{*}}$.
\end{lem}
Now using \eqref{asy 0 pi} in Lemma \ref{nfn}, we get as $z\rightarrow\infty$ within $\mathrm{Sect}_0$, for $1\leq i \leq n-1$, $1\leq j \leq n$,
\begin{align}
&F\left(z;E_n,A_{n-1}\right)_{i j} \sim \frac{a^{(n-1)}_i}{\lambda^{\left(n\right)}_j-\lambda^{\left(n-1\right)}_i}\left(\sum_{k=1}^{n-1}C_{i j,k}\frac{1}{\mathrm{e}^{-\frac{\pi \mathrm{i} }{2}a_{ij,k}}} z^{-a_{i j,k}}+C_{i j,n}\mathrm{e}^{\mathrm{i} z}\frac{1}{\mathrm{e}^{\frac{\pi \mathrm{i} }{2}a_{ij,n}}} z^{-a_{ij,n}}\right)z^{-\frac{1}{2\pi \mathrm{i}}\lambda^{(n)}_j},  \label{asymptotic using nFn 1}\\
    &F\left(z;E_n,A_{n-1}\right)_{n j}\sim \left(\sum_{k=1}^{n-1}C_{n j,k}\frac{1}{\mathrm{e}^{-\frac{\pi \mathrm{i} }{2}a_{nj,k}}} z^{-a_{n j,k}}+C_{n j,n}\mathrm{e}^{\mathrm{i} z} \frac{1}{\mathrm{e}^{\frac{\pi \mathrm{i} }{2}a_{nj,n}}} z^{-a_{nj,n}}\right)z^{-\frac{1}{2\pi \mathrm{i}}\lambda^{(n)}_j} \label{asymptotic using nFn 2},
\end{align}
where the constants $C_{i j,k}$, for $1\leq i,j \leq n$, are
\begin{align*}
    &C_{i j,k}=\frac{\prod_{l=1,l\ne j}^{n} \Gamma\left(b_{ij,l}\right)}{\prod_{l=1,l\ne k}^{n-1} \Gamma\left(a_{ij,l}\right)}\frac{\prod_{l=1,l\ne k}^{n-1} \Gamma\left(a_{ij,l}-a_{ij,k}\right)}{\prod_{l=1,l\ne j}^{n}\Gamma\left(b_{ij,l}-a_{ij,k}\right)}, \ \ 1\leq k \leq n-1, \\
    &C_{i j,n}=\frac{\prod_{l=1,l \ne j}^{n} \Gamma\left(b_{ij,l}\right)}{\prod_{l=1}^{n-1}\Gamma\left(a_{ij,l}\right)},
\end{align*}
and $a_{i j,n}$, for $1\leq i \leq n-1$, $1\leq j \leq n$, are
\begin{equation*}
    a_{i j,n}=\sum_{l=1,l\ne j}^{n}b_{ij,l}-\sum_{l=1}^{n-1}a_{ij,l}
    =1+\frac{1}{2\pi \mathrm{i} }\left(\lambda^{\left(n-1\right)}_n-\lambda^{\left(n\right)}_j\right).
\end{equation*}
Comparing the coefficients in the asymptotic behaviors \eqref{Fij asymptotic behavior}, \eqref{Fnj asymptotic behavior}, \eqref{asymptotic using nFn 1} and \eqref{asymptotic using nFn 2}, we can determine $U_0$ as follows.

\begin{itemize}
    \item 
    For $1\leq k \leq n-1$ and $1\leq j \leq n$, we find that $\left(U_0\right)_{k j}$ satisfies the following $n$ equations
\begin{align*}
    a^{\left(n-1\right)}_i \frac{1}{\lambda^{\left(n\right)}_j-\lambda^{\left(n-1\right)}_i}C_{i j, k}\mathrm{e}^{\frac{\pi \mathrm{i} }{2}a_{ij,k}}&=\left(K_1\right)_{i k} a^{\left(n-1\right)}_k \left(U_0\right)_{k j}, \quad 1\leq i\neq k \leq n-1,\\
    a^{\left(n-1\right)}_k \frac{1}{\lambda^{\left(n\right)}_j-\lambda^{\left(n-1\right)}_k}C_{k j, k}\mathrm{e}^{\frac{\pi \mathrm{i} }{2}a_{kj,k}}&= a^{\left(n-1\right)}_k \left(U_0\right)_{k j},\\
    C_{nj,k}\mathrm{e}^{\frac{\pi \mathrm{i} }{2}a_{nj,k}} &=\left(K_1\right)_{n k}a^{\left(n-1\right)}_k \left(U_0\right)_{k j},
\end{align*}
which imply that (using identities \eqref{F1} and \eqref{F2})
\begin{align*}
    \left(U_0\right)_{k j}=\frac{1}{\lambda^{\left(n\right)}_j-\lambda^{\left(n-1\right)}_k}\frac{\prod_{l=1,l\ne j}^{n} \Gamma\left(b_{kj,l}\right)}{\prod_{l=1,l\ne k}^{n-1} \Gamma\left(a_{kj,l}\right)}\frac{\prod_{l=1,l\ne k}^{n-1} \Gamma\left(a_{kj,l}-a_{kj,k}\right)}{\prod_{l=1,l\ne j}^{n}\Gamma\left(b_{kj,l}-a_{kj,k}\right)}\frac{1}{\mathrm{e}^{-\frac{\pi \mathrm{i} }{2}a_{kj,k}}}.
\end{align*}
\item 
For $1\leq j \leq n$, we find that $\left(U_0\right)_{n j}$ satisfies the following $n$ equations
\begin{align*}
    a^{\left(n-1\right)}_i \frac{1}{\lambda^{\left(n\right)}_j-\lambda^{\left(n-1\right)}_i}C_{i j,n} \mathrm{e}^{-\frac{\pi \mathrm{i} }{2}a_{ij,n}}=\left(K_1\right)_{i n} \left(U_0\right)_{nj}, \quad 1\leq i \leq n-1,
\end{align*}
and
\begin{align*}
    C_{n j,n}\mathrm{e}^{-\frac{\pi \mathrm{i} }{2}a_{nj,n}}=\left(U_0\right)_{n j},
\end{align*}
which imply that (using the identity \eqref{Fn})
\begin{align*}
    \left(U_0\right)_{nj}=\frac{\prod_{l=1,l \ne j}^{n} \Gamma\left(b_{nj,l}\right)}{\prod_{l=1}^{n-1}\Gamma\left(a_{nj,l}\right)}\frac{1}{\mathrm{e}^{\frac{\pi \mathrm{i} }{2}a_{nj,n}}}.
\end{align*}
\end{itemize}
We can deduce $U_{-1}$ in the same manner as $U_0$ using \eqref{asy -pi 0} in Lemma \ref{nfn}.
\end{proof}
\begin{rmk}
    In the above proof, we see that the asymptotics of the last row of the solution $F\left(z;E_n,A_{n-1}\right)$, given in \eqref{Fnj asymptotic behavior}, already determines the whole matrix $U_0$.
\end{rmk}

\begin{proposition}\label{S+ diag sys}
Under the assumption of Proposition \ref{classical solution}, the Stokes matrix $S_{+}\left(E_n,A_{n-1}\right)$ of the system \eqref{classical new equation} is upper triangular with the form
\begin{equation}\label{explicit S+}
S_{+}\left(E_{n},A_{n-1}\right)=\left(\begin{array}{cc}
    \mathrm{e}^{\frac{ A^{\left(n-1\right)}_{n-1}}{2}} & b_{+}^{\left(n\right)}  \\
    0 & \mathrm{e}^{\frac{a_{nn}}{2}}
  \end{array}\right),
  \end{equation}
where $A^{\left(n-1\right)}_{n-1}=\mathrm{diag}\left(\lambda^{\left(n-1\right)}_1,...,\lambda^{\left(n-1\right)}_{n-1}\right)$, $b^{(n)}_{+}=\left(\big(b^{(n)}_{+}\big)_1,...,\big(b^{(n)}_{+}\big)_{n-1}\right)^{\intercal}$ and
\begin{equation*}
\big(b^{(n)}_{+}\big)_{j }=\frac{\mathrm{e}^{\frac{\lambda_{j}^{\left(n-1\right)}+\lambda_{n}^{\left(n-1\right)}}{4}} \prod_{l=1}^{n-1} \Gamma\left(1+\frac{\lambda_{l}^{\left(n-1\right)}-\lambda_{j}^{\left(n-1\right)}}{2 \pi \mathrm{i} }\right)}{\prod_{l=1}^{n} \Gamma\left(1+\frac{\lambda_{l}^{\left(n\right)}-\lambda_{j}^{\left(n-1\right)}}{2 \pi \mathrm{i} }\right)} \cdot a_{j}^{\left(n-1\right)}.
\end{equation*}
Here we recall the symbol $\lambda^{\left(n-1\right)}_n=a_{n n}$. The matrix $S_{-}\left(E_n,A_{n-1}\right)$ is lower triangular and can be computed in a similar way.
\end{proposition}
\begin{proof}
The Stokes matrix $S_+\left(E_n,A_{n-1}\right)$ of the system \eqref{classical new equation} is the unique matrix such that on the sector $z\in S\left(-\pi,0\right)$, we have $F_{-1}\left(z;E_n, A_{n-1}\right)\cdot \mathrm{e}^{-\frac{\delta_{n-1}\left(A\right)}{2}} S_{+}\left(E_n, A_{n-1}\right)=F_0\left(z;E_n, A_{n-1}\right) $. Then by Proposition \ref{mainpro}, we have 
\begin{equation}\label{S+ eq}
     S_{+}\left(E_n,A_{n-1}\right) Y U_0=\mathrm{e}^{\frac{\delta_{n-1}\left(A\right)}{2}} Y U_{-1}.
\end{equation} 
Plugging \eqref{explicit S+} into \eqref{S+ eq}, we find that $\left(b^{\left(n\right)}_{+}\right)_{j }$ satisfies $n$ equations as follows
\begin{equation*}
    \mathrm{e}^{\frac{\lambda^{\left(n-1\right)}_j}{2}} a^{\left(n-1\right)}_j \left(U_{0}\right)_{j k}+\big(b^{(n)}_{+}\big)_{j } \left(U_0\right)_{n k}=\mathrm{e}^{\frac{\lambda^{\left(n-1\right)}_j}{2}} a^{\left(n-1\right)}_j \left(U_{-1}\right)_{j k}, \quad 1\leq k \leq n,
\end{equation*}
which imply that
\begin{equation*}
    \big(b^{(n)}_{+}\big)_{j }=\mathrm{e}^{\frac{\lambda^{\left(n-1\right)}_j}{2}} a^{\left(n-1\right)}_j\left(\left(U_{-1}\right)_{j k}-\left(U_{0}\right)_{j k}\right)\left(U_0\right)_{n k}^{-1}.
\end{equation*}
From Proposition \ref{mainpro}, we have
\begin{equation}\label{U0 nk}
\begin{aligned}
    \left(U_0\right)_{n k}^{-1}
    &=\frac{\Gamma\left(1+\frac{1}{2\pi \mathrm{i} }\left(\lambda^{\left(n-1\right)}_j-\lambda^{\left(n\right)}_k\right)\right)\prod_{l=1,l\neq j}^{n-1}\Gamma\left(a_{nk,l}\right)}{\prod_{l=1,l \ne k}^{n} \Gamma\left(b_{nk,l}\right)}\mathrm{e}^{\frac{\pi \mathrm{i} }{2}a_{nk,n}}\\
    &=\frac{1}{2\pi \mathrm{i} }\frac{\Gamma\left(\frac{1}{2\pi \mathrm{i} }\left(\lambda^{\left(n-1\right)}_j-\lambda^{\left(n\right)}_k\right)\right)\prod_{l=1,l\neq j}^{n-1}\Gamma\left(a_{n k,l}\right)}{\prod_{l=1,l \ne k}^{n} \Gamma\left(b_{nk,l}\right)}\left(\lambda^{\left(n-1\right)}_j-\lambda^{\left(n\right)}_k\right)\mathrm{e}^{\frac{\pi \mathrm{i} }{2}a_{nk,n}}\\
    &=\frac{1}{2\pi \mathrm{i} }\frac{\pi \prod_{l=1,l\neq j}^{n-1}\Gamma\left(a_{n k,l}\right)\left(\lambda^{\left(n-1\right)}_j-\lambda^{\left(n\right)}_k\right)\mathrm{e}^{\frac{\pi \mathrm{i} }{2}a_{nk,n}}}{\mathrm{sin}\left(\frac{1}{2 \mathrm{i} }\left(\lambda^{\left(n-1\right)}_j-\lambda^{\left(n\right)}_k\right)\right)\Gamma\left(1-\frac{1}{2\pi \mathrm{i} }\left(\lambda^{\left(n-1\right)}_j-\lambda^{\left(n\right)}_k\right)\right)\prod_{l=1,l \ne k}^{n} \Gamma\left(b_{nk,l}\right)},
\end{aligned}
\end{equation}
and 
\begin{equation}
\begin{aligned}\label{U1-U0}
    \left(U_{-1}\right)_{j k}-\left(U_{0}\right)_{j k}
    &=\frac{1}{\mathrm{e}^{\frac{\pi\mathrm{i} }{2}a_{jk,j}}}\frac{\frac{1}{\mathrm{e}^{\pi \mathrm{i} a_{jk,j}}}-\frac{1}{\mathrm{e}^{-\pi \mathrm{i} a_{jk,j}}}}{\lambda^{\left(n\right)}_k-\lambda^{\left(n-1\right)}_j}\frac{\prod_{l=1,l\ne k}^{n} \Gamma\left(b_{jk,l}\right)}{\prod_{l=1,l\ne j}^{n-1} \Gamma\left(a_{jk,l}\right)}\frac{\prod_{l=1,l\ne j}^{n-1} \Gamma\left(a_{jk,l}-a_{jk,j}\right)}{\prod_{l=1,l\ne k}^{n}\Gamma\left(b_{jk,l}-a_{jk,j}\right)}\\
    &=\frac{1}{ \mathrm{e}^{\frac{\pi\mathrm{i} }{2}a_{jk,j}}} \frac{-2\mathrm{i} \cdot\mathrm{sin}\left(a_{j k,j}\pi\right)}{\lambda^{\left(n\right)}_k-\lambda^{\left(n-1\right)}_j}\frac{\prod_{l=1,l\ne k}^{n} \Gamma\left(b_{jk,l}\right)}{\prod_{l=1,l\ne j}^{n-1} \Gamma\left(a_{jk,l}\right)}\frac{\prod_{l=1,l\ne j}^{n-1} \Gamma\left(a_{jk,l}-a_{jk,j}\right)}{\prod_{l=1,l\ne k}^{n}\Gamma\left(b_{jk,l}-a_{jk,j}\right)}.
\end{aligned}
\end{equation}
Here the last row of \eqref{U0 nk} is derived from the Euler’s reflection formula
\begin{equation*}
    \Gamma\left(1-\frac{1}{2\pi \mathrm{i} }\left(\lambda^{\left(n-1\right)}_j-\lambda^{\left(n\right)}_k\right)\right)\Gamma\left(\frac{1}{2\pi \mathrm{i} }\left(\lambda^{\left(n-1\right)}_j-\lambda^{\left(n\right)}_k\right)\right)=\frac{\pi}{\sin\left(\frac{1}{2 \mathrm{i} }\left(\lambda^{\left(n-1\right)}_j-\lambda^{\left(n\right)}_k\right)\right)}
\end{equation*}
and the last row of \eqref{U1-U0} is derived from the Euler's formula
\begin{equation*}
    \frac{1}{\mathrm{e}^{\pi \mathrm{i} a_{jk,j}}}-\frac{1}{\mathrm{e}^{-\pi \mathrm{i} a_{jk,j}}}=-2\mathrm{i} \cdot\mathrm{sin}\left(a_{j k,j}\pi\right).
\end{equation*}
Combining \eqref{U0 nk} and \eqref{U1-U0}, we have
\begin{equation*}
\big(b^{(n)}_{+}\big)_{j }=\frac{\mathrm{e}^{\frac{\lambda_{j}^{\left(n-1\right)}+\lambda_{n}^{\left(n-1\right)}}{4}} \prod_{l=1}^{n-1} \Gamma\left(1+\frac{\lambda_{l}^{\left(n-1\right)}-\lambda_{j}^{\left(n-1\right)}}{2 \pi \mathrm{i} }\right)}{\prod_{l=1}^{n} \Gamma\left(1+\frac{\lambda_{l}^{\left(n\right)}-\lambda_{j}^{\left(n-1\right)}}{2 \pi \mathrm{i} }\right)} \cdot a_{j}^{\left(n-1\right)}.
\end{equation*}
\end{proof}

\subsection{Extension to all nonresonant $A$}\label{secextension}
It follows from Proposition \ref{10} that
under the assumption of Proposition \ref{classical solution}, we have
    \begin{equation}\label{Sexp}
        S_{+}\left(E_n,A\right)=\operatorname{diag}\left({P}_{n-1},1\right)S_{+}\left(E_n,A_{n-1}\right)\operatorname{diag}\left({P}_{n-1}^{-1},1\right).
    \end{equation}
The $S_{-}\left(E_n,A\right)$ case is similar.
However, by analytic continuation, if we just require that $\frac{1}{2\pi \mathrm{i} }\left(\lambda^{\left(n\right)}_i-\lambda^{\left(n\right)}_j\right)$ is either zero or not an integer, (i.e., whenever the Stokes matrices are well defined), we still have \eqref{Sexp}.

\begin{proposition}\label{main stokes}
Under the nonresonant condition \eqref{nonres}, the Stokes matrix $S_{+}\left(E_{n},A\right)$ of the system \eqref{classical original equation} takes the block matrix form
\begin{equation*}
S_{+}\left(E_{n},A\right)=\left(\begin{array}{cc}
    \mathrm{e}^{\frac{A^{\left(n-1\right)}}{2}} & b_{+}  \\
    0 & \mathrm{e}^{\frac{a_{nn}}{2}}
  \end{array}\right),
  \end{equation*}
where $b_{+}=\left(\left(b_+\right)_1,...,\left(b_{+}\right)_{n-1}\right)^{\intercal}$ with
\begin{equation}\label{b expli by pr}
\left(b_{+}\right)_k=
\sum_{i=1}^{n-1}\sum_{j=1}^{n-1} \frac{\mathrm{e}^{\frac{\lambda_{i}^{\left(n-1\right)}+a_{nn}}{4}} \prod_{l=1}^{n-1} \Gamma\left(1+\frac{\lambda_{l}^{\left(n-1\right)}-\lambda_{i}^{\left(n-1\right)}}{2 \pi \mathrm{i} }\right)}{\prod_{l=1}^{n} \Gamma\left(1+\frac{\lambda_{l}^{\left(n\right)}-\lambda_{i}^{\left(n-1\right)}}{2 \pi \mathrm{i} }\right)} \cdot \left(\mathrm{pr}_i\right)_{k j} \cdot a_{jn}.
\end{equation}
Here for $1\leq i \leq n-1$, $\mathrm{pr}_i$ is a matrix as follows
\begin{equation*}
    \mathrm{pr}_i=\prod_{l=1,l \neq i}^{n-1} \frac{A^{(n-1)}-\lambda^{(n-1)}_l\cdot \mathrm{Id}_{n-1}}{\lambda^{(n-1)}_i-\lambda^{(n-1)}_l}.
\end{equation*}
\end{proposition}
\begin{proof}
From direct computation, the explicit expression of $\operatorname{diag}\left({P}_{n-1},1\right)S_{+}\left(E_n,A\right)\operatorname{diag}\left({P}_{n-1}^{-1},1\right)$ tells us that
\begin{equation}\label{b+k}
    \left(b_{+}\right)_k=\sum_{i=1}^{n-1}\sum_{j=1}^{n-1} \frac{\mathrm{e}^{\frac{\lambda_{i}^{\left(n-1\right)}+a_{nn}}{4}} \prod_{l=1}^{n-1} \Gamma\left(1+\frac{\lambda_{l}^{\left(n-1\right)}-\lambda_{i}^{\left(n-1\right)}}{2 \pi \mathrm{i} }\right)}{\prod_{l=1}^{n} \Gamma\left(1+\frac{\lambda_{l}^{\left(n\right)}-\lambda_{i}^{\left(n-1\right)}}{2 \pi \mathrm{i} }\right)} \cdot \frac{\left(\mathcal{P}_{n-1}\right)_{ki}\left(\mathcal{Q}_{n-1}\right)_{ij}}{\left(D_{n-1}\right)_{i}} \cdot a_{jn}.
\end{equation}
From Cayley-Hamilton theorem in linear algebra, we have
\begin{equation*}
    \prod_{l=1}^{n-1}\left(A^{(n-1)}-\lambda^{(n-1)}_l\cdot \mathrm{Id}_{n-1}\right)=0.
\end{equation*}
Then we have
\begin{equation}\label{APr PrA}
    A^{(n-1)} \mathrm{pr}_i=\mathrm{pr}_i A^{(n-1)}=\lambda^{(n-1)}_i \mathrm{pr}_i,
\end{equation}
and by division with remainder,
\begin{equation}\label{PrPrPr}
    \mathrm{pr}_i \cdot \mathrm{pr}_i=\mathrm{pr}_i.
\end{equation}
From \eqref{APr PrA} and \eqref{PrPrPr}, every column (row) of $\mathrm{pr}_i$ is column (row) eigenvector with eigenvalue $\lambda^{\left(n-1\right)}_i$ and $\mathrm{pr}_i$ has the following form
\begin{equation}\label{pri}
\mathrm{pr}_{i}=\left(\left(\mathcal{P}_{n-1}\right)_{1i},...,\left(\mathcal{P}_{n-1}\right)_{n-1,i}\right)^{\intercal}\cdot \frac{1}{\left(D_{n-1}\right)_i} \cdot \left(\left(\mathcal{Q}_{n-1}\right)_{i1},...,\left(\mathcal{Q}_{n-1}\right)_{i,n-1}\right).
\end{equation}
Using \eqref{pri}, the identity \eqref{b+k} can be rewritten as \eqref{b expli by pr}. 
\end{proof}

We note that the expression $\mathrm{diag}\left(P_{n-1},1\right) F\left(z;E_n,A_{n-1}\right) \mathrm{diag}\left(P_{n-1}^{-1},1\right)$, with $F\left(z;E_n,A_{n-1}\right)$ the solution of the diagonalized equation \eqref{classical new equation} given as in Proposition \ref{classical solution}, smoothly extends from $A$ satisfying \eqref{nonres}, \eqref{ineq1}, \eqref{ineq2}, to those $A$ satisfying the nonresonant condition \eqref{nonres},
and gives the fundamental solution of the equation \eqref{classical original equation}. The Stokes matrices $S_{\pm}\left(E_{n},A\right)$ are analytic function with respect to all $A$ satisfying the nonresonant condition \eqref{nonres}, which also follows from the definition of the Stokes matrices of the nonresonant equation.

\section{The explicit expression of Stokes matrices of quantum equation \texorpdfstring{\eqref{introqeq}}{d}}\label{beginsec}
This section computes explicitly the quantum Stokes matrices of the linear system \eqref{introqeq}. 
Its structure is parallel to Section \ref{sec2}. In particular, Section \ref{qdiff and qstokes} introduces the Stokes data of the quantum linear systems \eqref{introqeq}. Section \ref{sec:qminor}--\ref{sec:diag}
introduces a gauge transformation that diagonalizes the upper left block of the quantum equation \eqref{introqeq}.  Section \ref{sec:qkummer} studies asymptotics of the solution of the equation after gauge transformation and then derives the explicit expression of its quantum Stokes matrices via the asymptotics.
Section \ref{qsecextension} then obtains the quantum Stokes matrices of the quantum equation \eqref{introqeq}.

\subsection{The Stokes matrices of quantum  equation \eqref{introqeq}}
\label{qdiff and qstokes}
\begin{proposition}\label{uniformal}
For any nonzero real number $h$, the ordinary differential equation \eqref{introqeq}
\begin{eqnarray*}
\frac{\mathrm{d}F}{\mathrm{d}z}=h\left(\mathrm{i} E_n+\frac{1}{2\pi \mathrm{i} }\frac{T}{z}\right)\cdot F,
\end{eqnarray*}
has a unique formal fundamental solution of the form \begin{eqnarray}\label{formalsum}
\hat{F}_h\left(z;E_n,T\right)=\hat{H}\left(z\right) \mathrm{e}^{{h\mathrm{i} E_nz}}z^{\frac{h\delta_{n-1}\left(T\right)}{2\pi\mathrm{i}}}, \ \ \ {\it for} \ \hat{H}=\mathrm{Id}+H_1z^{-1}+H_2z^{-2}+\cdot\cdot\cdot, \end{eqnarray}
where each coefficient $H_m\in{\rm End}\left(L\left(\lambda\right)\right)\otimes{\rm End}\left(\mathbb{C}^n\right)$, and 
\begin{equation}\label{qdeltan-1}
    \delta_{n-1}\left(T\right)_{ij}=\left\{
          \begin{array}{lr}
             T_{ij},   & \text{if} \ \ 1\le i, j\le n-1, \ \text{or} \ i=j=n ; \\
           0, & \text{otherwise}.
             \end{array}
\right.
\end{equation}
\end{proposition}
\begin{proof}
Plugging $\hat{F}_h$ into the equation \eqref{introqeq} gives rise to the equation for $\hat{H},$
\begin{eqnarray*}
\frac{1}{h}\frac{\mathrm{d}\hat{H}}{\mathrm{d}z}+\hat{H}\cdot \left(\mathrm{i} E_n+
\frac{\delta_{n-1}\left(T\right)}{2\pi\mathrm{i} z}\right)=\left( \mathrm{i} E_n +
\frac{T}{2\pi\mathrm{i} z}\right)\cdot \hat{H}.\end{eqnarray*}
Comparing the coefficients of $z^{-m-1}$, we see that $H_m$ satisfies 
\begin{eqnarray}\label{simHm}
-2\pi \left[H_{m+1}, E_n\right]=\left(\frac{2\pi\mathrm{i} m}{h}+T\right)\cdot H_{m}-H_{m}\cdot   \delta_{n-1}\left(T\right).\end{eqnarray}
Set $\{E_{kl}\}_{1\le k,l\le n}$ the standard basis of ${\rm End}\left(\mathbb{C}\right)$. Then \[T=\sum_{k,l} e_{kl}\otimes E_{kl}, \ \text{ and }\  E_n= 1\otimes E_{nn}.\] Plugging $H_m=\sum_{k,l} H_{m, kl}\otimes E_{kl}$, with $H_{m,kl}\in{\rm End}\left(L\left(\lambda\right)\right)$, into the equation \eqref{simHm} gives rise to 
\begin{eqnarray*}
&&\sum_{k,l}\left(u_l-u_k\right) H_{m+1, kl}\otimes E_{kl}\\
&=&\sum_{k,l}\frac{2\pi \mathrm{i} m}{h}H_{m,kl}\otimes E_{kl}+
\sum_{k,l,l'} e_{kl'} H_{m, l'l} \otimes E_{kl}-\sum_{k,l,l'}   H_{m, kl'} \delta_{n-1}\left(T\right)_{l'l}\otimes E_{kl}.
\end{eqnarray*}
Here $e_{kl}$ is understood as an element in ${\rm End}\left(L\left(\lambda\right)\right)$ via the given representation and $u_i=-2\pi \delta_{n i}$, for $1\leq i \leq n$. That is for $1\le k,l\le n-1$, or $k=l=n$,
\begin{equation}
    \label{Hknel}
0=\frac{2\pi\mathrm{i} m}{h}H_{m,kl}+\sum_{s=1}^n e_{ks} H_{m, sl}- \left(1-\delta_{n l}\right)\sum_{t=1}^{n-1}H_{m, kt} e_{tl}-\delta_{n l}H_{m,k n}e_{n n} \ \in {\rm End}\left(L\left(\lambda\right)\right), 
\end{equation}
and for $1\leq k \leq n-1$,
\begin{align}
&-2\pi H_{m+1,kn}=\frac{2\pi\mathrm{i} m}{h}H_{m,kn}+\sum_{s=1}^n e_{ks} H_{m, sn}- H_{m, kn} e_{nn} \ \in {\rm End}\left(L\left(\lambda\right)\right),\label{Hknel1} \\
&2\pi H_{m+1,nk}=\frac{2\pi\mathrm{i} m}{h}H_{m,nk}+\sum_{s=1}^n e_{ns} H_{m, sk}- \sum_{t=1}^{n-1}H_{m, nt} e_{tk} \ \in {\rm End}\left(L\left(\lambda\right)\right).\label{Hknel2}
\end{align}
To see the above recursive relations have a unique solution, first note that given $H_m$, the identities \eqref{Hknel1} and \eqref{Hknel2} uniquely define the "block off-diagonal" part $H_{m+1, kn}$ and $H_{m+1, nk}$ of $H_{m+1}$ from $H_m$. Furthermore, the identity \eqref{Hknel} (replacing $m$ by $m+1$) is 
\[0=\sum_{1\le k,l\le n-1 } e_{kn} H_{m+1, nl}\otimes E_{kl}+\sum_{t=1}^{n-1}e_{n t}H_{m+1,tn}\otimes E_{n n}+\frac{2\pi\mathrm{i} \left(m+1\right)}{h}\delta_{n-1}\left(H_{m+1}\right)+\left[\delta_{n-1}\left(T\right),\delta_{n-1}\left(H_{m+1}\right)\right],\]
where 
\[\delta_{n-1}\left(H_{m+1}\right)_{ij}=\left\{
          \begin{array}{lr}
             \left(H_{m+1}\right)_{ij},   & \text{if} \ \ 1\le i, j\le n-1, \ \text{or} \ i=j=n  \\
           0, & \text{otherwise}.
             \end{array}
\right.\]
Since $h$ is a real number (actually as long as $h\notin \mathrm{i} \mathbb{Q}$), we have that $2\pi \mathrm{i} \left(m+1\right)/h{\rm Id}+{\rm ad}_{\delta_{n-1}\left(T\right)}$ is invertible on the subspace \[\mathop {\oplus}_{1\le k,l\le n-1, k=l=n} {\rm End}\left(L\left(\lambda\right)\right)\otimes E_{kl}\subset {\rm End}\left(L\left(\lambda\right)\right)\otimes {\rm End}\left(\mathbb{C}^n\right)\] for any nonzero integer $m+1$. Thus, the "block diagonal" part $\{H_{m+1, kl}\}_{1\le k,l\le n-1 , k=l=n}$ of $H_{m+1}$ is uniquely determined from the block off-diagonal part. \end{proof} 

Again, the following proposition follows from the general principal of differential equations with irregular singularities. Here we take the notation of the Stokes sectors from Definition \ref{defSect}.

\begin{proposition}\label{qsol}
For $i\in \mathbb{Z}$, there exists a unique analytic solution $F_{h,i}\left(z;E_n,T\right)$ with the asymptotic expansion
\begin{align*}
    F_{h,i}\left(z;E_n,T\right)\cdot e^{-{h\mathrm{i} E_nz}}z^{-\frac{h\delta_{n-1}\left(T\right)}{2\pi\mathrm{i}}} \sim \hat{H}\left(z\right),\quad z \rightarrow \infty \text{ in $\operatorname{Sect}_i$}.
\end{align*}
\end{proposition}

\begin{definition}
For fixed nonzero real number $h$, the quantum Stokes matrices $S_{h \pm}\left(E_n, T\right)$ of the system {\eqref{introqeq}} are the unique matrices such that
    \begin{itemize}
        \item on the sector $z\in S\left(d_0-\frac{\pi}{2},d_0+\frac{\pi}{2}\right)$, we have $F_{h,-1}\left(z;E_n,T\right)\cdot \mathrm{e}^{\frac{h\delta_{n-1}\left(T\right)}{2}} S_{h+}\left(E_n, T\right) =F_{h,0}\left(z;E_n,T\right)$, and
        \item on the sector $z\in S\left(d_{1}-\frac{\pi}{2},d_{1}+\frac{\pi}{2}\right)$, we have $F_{h,0}\left(z;E_n,T\right)\cdot S_{h-}\left(E_n, T\right)\mathrm{e}^{-\frac{h\delta_{n-1}\left(T\right)}{2}}=F_{h,-1}\left(z\mathrm{e}^{-2 \pi \mathrm{i} };E_n,T\right)$.
    \end{itemize}
\end{definition}

\subsection{Quantum minors and Gelfand-Tsetlin basis}\label{sec:qminor}

In this subsection, we recall quantum minors and Gelfand-Tsetlin basis. See more discussions in \cite{GT, Molev89, Molev}. In particular, the following propositions can be found in \cite{Molev89, Molev}. 


\begin{definition}\label{qminor}
For an indeterminate $x$ commuting with all $e_{ij}$, $1\leq i,j \leq n$, let us introduce the $n\times n$ matrix $T\left(x\right)=T-x \operatorname{Id}$. For $1\leq m \leq n$, given two sequences, $a=\{a_1,\cdots,a_m\}$ and $b=\{b_1,\cdots,b_m\}$, with elements in $\{1,2,\cdots,n-1,n\}$, the corresponding quantum
minor of the matrix $T\left(x\right)$ is defined by 
\begin{align*}
\Delta^{a_1,...,a_m}_{b_1,...,b_m}\left(T\left(x\right)\right):=\sum_{\sigma\in S_m}\left(-1\right)^\sigma T_{a_{\sigma\left(1\right)},b_1}\left(x_1\right)\cdots T_{a_{\sigma\left(m\right)},b_m}\left(x_m\right)\in U\left({\mathfrak {gl}}_n\right)[x],
\end{align*}
where $x_k:=x+k-1$ for $k=1,...,m$, and $\left(-1\right)^{\sigma}$ means the signature of the permutation $\sigma$ in the symmetry group $S_m$ of $m$ elements.
\end{definition}

\begin{proposition}\label{Laplace expansion}
Given two sequences $a=\{a_1,\cdots,a_m\}$ and $b=\{b_1,\cdots,b_m\}$, 
we have the following Laplace expansion formulas
\begin{align*}
&\Delta^{a_1,...,a_m}_{b_1,...,b_m}\left(T\left(x\right)\right)=\sum_{k=1}^{m}\left(-1\right)^{k+j} T_{a_{k} b_{j}}\left(x_{1}\right) \Delta_{b_1,...,\widehat{b_j},...,b_m}^{a_1,...,\widehat{a_k},...,a_m}\left(T\left(x_{2}\right)\right), \quad 1\leq j \leq m,\\
&\Delta^{a_1,...,a_m}_{b_1,...,b_m}\left(T\left(x\right)\right)=\sum_{k=1}^{m}\left(-1\right)^{k+j} \Delta_{b_1,...,\widehat{b_j},...,b_m}^{a_1,...,\widehat{a_k},...,a_m}\left(T\left(x_{1}\right)\right) T_{a_{k} b_{j}}\left(x_{m}\right), \quad 1\leq j \leq m,\\
&\Delta^{a_1,...,a_m}_{b_1,...,b_m}\left(T\left(x\right)\right)=\sum_{j=1}^{m}\left(-1\right)^{k+j} T_{a_{k} b_{j}}\left(x_{m}\right) \Delta_{b_1,...,\widehat{b_j},...,b_m}^{a_1,...,\widehat{a_k},...,a_m}\left(T\left(x_{1}\right)\right), \quad 1\leq k \leq m, \\
&\Delta^{a_1,...,a_m}_{b_1,...,b_m}\left(T\left(x\right)\right)=\sum_{j=1}^{m}\left(-1\right)^{k+j} \Delta_{b_1,...,\widehat{b_j},...,b_m}^{a_1,...,\widehat{a_k},...,a_m}\left(T\left(x_{2}\right)\right) T_{a_{k} b_{j}}\left(x_{1}\right), \quad 1\leq k \leq m.
\end{align*}
Here $\Delta^{a_1,...,\widehat{a_k},..., a_m}_{b_1,...,\widehat{b_j},...,b_{m}}\left(T\left(x\right)\right)$ is the $\left(m-1\right)\times \left(m-1\right)$ quantum minor of the matrix $T\left(x\right)$ (the $\widehat{a_k}$ or $\widehat{b_j}$ means that the row or column index $a_k$ or $b_j$ is omitted).
\end{proposition}
\begin{proposition}{\label{2}}
The element $e_{ij}$ and quantum minor $\Delta^{a_1,...,a_m}_{b_1,...,b_m}\left(T\left(x\right)\right)$ have the following relation
\begin{align*}
    \left[e_{i j}, \Delta^{a_1,...,a_m}_{b_1,...,b_m}\left(T\left(x\right)\right)\right]=\sum_{r=1}^{m}\left(\delta_{j a_{r}} \Delta^{a_1,...,i...,a_m}_{b_1,...,b_m}\left(T\left(x\right)\right)-\delta_{i b_{r}}\Delta^{a_1,...,a_m}_{b_1,...,j...,b_m}\left(T\left(x\right)\right) \right),
\end{align*}
where $i$ and $j$ on the right hand side take the $r$-th position. As a result, $\Delta^{1,...,n}_{1,...,n}\left(T\left(x\right)\right)$ commutes with every $e_{ij}$, for $1\leq i,j \leq n$.
\end{proposition} 
Now let us introduce the Gelfand-Tsetlin basis and the actions of the quantum minors on them. The representation $L\left(\lambda\right)$ has a set of orthogonal basis vectors called the Gelfand–Tsetlin basis vectors parameterized by the Gelfand–Tsetlin patterns. See \cite{Molev} for more details. Such a pattern $\Lambda$ is a collection of numbers $\{\lambda^{\left(i\right)}_j\left(\Lambda\right)\}_{1\le j\le i\le n}$, with the $n$-tuples $\left(\lambda^{\left(n\right)}_{1}\left(\Lambda\right),\cdots,\lambda^{\left(n\right)}_{n}\left(\Lambda\right)\right)$ coinciding with $\lambda$ and satisfies the interlacing conditions
\begin{equation*}
\lambda_j^{\left(i\right)}\left(\Lambda\right)-\lambda^{\left(i-1\right)}_j\left(\Lambda\right)\in\mathbb{Z}_{\ge 0}, \hspace{5mm} \lambda^{\left(i-1\right)}_j\left(\Lambda\right)-\lambda^{\left(i\right)}_{j+1}\left(\Lambda\right)\in\mathbb{Z}_{\ge 0}, \hspace{5mm} \forall \ i=2,...,n.
\end{equation*}
The Gelfand–Tsetlin basis vector corresponding to the pattern $\Lambda$ is denoted as $\xi_{\Lambda}$. 

\begin{proposition}[see e.g., \cite{Molev89}]{\label{3}}
The actions of $e_{kk}$, $e_{k,k+1}$ and $e_{k+1,k}$ on the Gelfand-Tsetlin basis vector $\xi_{\Lambda}$ of the representation $L\left(\lambda\right)$ are given by the formulas
\begin{align*}
&e_{k k} \xi_{\Lambda} =\left(\sum_{j=1}^{k} \lambda^{\left(k\right)}_ j\left(\Lambda\right)-\sum_{j=1}^{k-1} \lambda^{\left(k-1\right)}_ j\left(\Lambda\right)\right) \xi_{\Lambda},\quad 1\leq k \leq n, \\
&e_{k, k+1} \xi_{\Lambda} =\sum_{j=1}^{k}\sqrt{-\frac{\prod_{i=1}^{k+1}\left(l_{k j}\left(\Lambda\right)-l_{k+1,i}\left(\Lambda\right)\right)\prod_{i=1}^{k-1}\left(l_{k j}\left(\Lambda\right)+1-l_{k-1,i}\left(\Lambda\right)\right)}{\prod_{i=1,i\neq j}^{k}\left(l_{k j}\left(\Lambda\right)-l_{k i}\left(\Lambda\right)\right)\prod_{i=1,i\neq j}^k \left(l_{k j}\left(\Lambda\right)+1-l_{k i}\left(\Lambda\right)\right)}}  \xi_{\Lambda+\delta_{k j}},\quad 1\leq k \leq n-1, \\
&e_{k+1, k} \xi_{\Lambda} =\sum_{j=1}^{k} \sqrt{-\frac{\prod_{i=1}^{k-1}\left(l_{k j}\left(\Lambda\right)-l_{k-1,i}\left(\Lambda\right)\right)\prod_{i=1}^{k+1}\left(l_{k j}\left(\Lambda\right)-1-l_{k+1,i}\left(\Lambda\right)\right)}{\prod_{i=1,i\neq j}^{k}\left(l_{k j}\left(\Lambda\right)-l_{k i}\left(\Lambda\right)\right)\prod_{i=1,i\neq j}^k\left(l_{k j}\left(\Lambda\right)-1-l_{k i}\left(\Lambda\right)\right)}}  \xi_{\Lambda-\delta_{k j}},\quad 1\leq k \leq n-1,
\end{align*}
where $l_{k j}\left(\Lambda\right)=\lambda^{\left(k\right)}_ j\left(\Lambda\right)-j+1$ are functions on the set of patterns, and $\Lambda \pm \delta_{k j}$ is a pattern replacing $\lambda^{\left(k\right)}_ j\left(\Lambda\right)$ of the pattern $\Lambda$ by $\lambda^{\left(k\right)}_ j\left(\Lambda\right)\pm 1$. If $\Lambda \pm \delta_{k j}$ is not a Gelfand–Tsetlin pattern, the vector $\xi_{\Lambda\pm \delta_{k j}}$ is thought as zero. 
\end{proposition}
\begin{rmk}
    The orthogonal Gelfand-Tsetlin basis $\{\xi_{\Lambda}\}$ in \cite{Molev89} is different from the orthogonal one in \cite{Molev}. If we write the orthogonal one in \cite{Molev} as $\{\eta_{\Lambda}\}$, then $\xi_{\Lambda}=\eta_{\Lambda}/\left|\eta_{\Lambda}\right|$, where $\left|\eta_{\Lambda}\right|$ is the norm of $\eta_{\Lambda}$ introduced in \cite[Proposition 2.4]{Molev}. Here we can derive Proposition \ref{3} from \cite[Theorem 2.3]{Molev} and \cite[Proposition 2.4]{Molev}.
\end{rmk}

\begin{definition}\label{ell}
For all $1\leq j \leq m\leq n$,  let $\ell^{\left(m\right)}_j$ be the endomorphism on $L\left(\lambda\right)$ defined by \[\ell^{\left(m\right)}_j\left(\xi_{\Lambda}\right)=l_{m j}\left(\Lambda\right)\xi_{\Lambda}\] for all basis elements $\xi_\Lambda$.
\end{definition}
\begin{proposition}
For all $1\leq m \leq n$, the action of $\Delta^{1,...,m}_{1,...,m}\left(T\left(x\right)\right)$ on the Gelfand-Tsetlin basis vector $\xi_{\Lambda}$ is
\begin{align*}
\Delta^{1,...,m}_{1,...,m}\left(T\left(x\right)\right)\xi_{\Lambda} =\left(\ell_1^{\left(m\right)}-x\right) \cdots\left(\ell_{m}^{\left( m\right)}-x\right) \xi_{\Lambda}.
\end{align*}
\end{proposition}


From Proposition \ref{2}, we know that $\ell^{(n-1)}_j$ does not commute with all $e_{i n}$, for $1\leq i, j \leq n-1$. So $\Delta^{1,\ldots, n-2,n-1}_{1,\ldots, n-2,n}\big(T(\ell^{(n-1)}_j)\big)$ is not well defined. However,

\begin{definition}
Suppose that $\Delta^{1,\ldots, n-2,n-1}_{1,\ldots, n-2,n}\left(T\left(x\right)\right)=\sum_{i=0}^{n-1}r_i x^i$, with coefficient $r_i\in U\left({\mathfrak {gl}}_n\right)$, then we can define $\left(\Delta_l\right)^{1,\ldots, n-2,n-1}_{1,\ldots, n-2,n}\big(T(\ell^{(n-1)}_j)\big)$ and $\left(\Delta_r\right)^{1,\ldots, n-2,n}_{1,\ldots, n-2,n-1}\big(T(\ell^{(n-1)}_j)\big)$ in the following way
\begin{align*}
&\left(\Delta_l\right)^{1,\ldots, n-2,n-1}_{1,\ldots, n-2,n}\big(T(\ell^{(n-1)}_j)\big)=\big(\ell^{(n-1)}_j\big)^{n-1} r_{n-1}+\big(\ell^{(n-1)}_j\big)^{n-2} r_{n-2}+\cdots +\ell^{\left(n-1\right)}_jr_{1}+r_0,\\
&\left(\Delta_r\right)^{1,\ldots, n-2,n}_{1,\ldots, n-2,n-1}\big(T(\ell^{(n-1)}_j)\big)=r_{n-1} \big(\ell^{(n-1)}_j\big)^{n-1}+ r_{n-2} \big(\ell^{(n-1)}_j\big)^{n-2}+\cdots +r_{1} \ell^{\left(n-1\right)}_j+r_0.
\end{align*}
\end{definition}

\begin{lem}\label{4}
The actions of $\left(\Delta_{l}\right)^{1,...,n-2,n-1}_{1,...,n-2,n}\big(T(\ell^{(n-1)}_j)\big)$ and $\left(\Delta_r\right)^{1,...,n-2,n}_{1,...,n-2,n-1}\big(T(\ell^{(n-1)}_j)\big)$ on the Gelfand-Tsetlin basis vector $\xi_{\Lambda}$ are as follows
\begin{align*}
    &\left(\Delta_{l}\right)^{1,...,n-1}_{1,...,n-2,n}\big(T(\ell^{(n-1)}_j)\big)\xi_{\Lambda}\\
    =&(-1)^{1+n+j}\sqrt{-\frac{\prod_{i} (l_{n-1,j}-l_{n i})\prod_{k\neq j}(l_{n-1,j}+1-l_{n-1,k})\prod_{s}(l_{n-1,j}+1-l_{n-2,s})}{\prod_{k\neq j}(l_{n-1,j}-l_{n-1,k})}}\xi_{\Lambda+\delta_{n-1,j}},
    \\
     &\left(\Delta_r\right)^{1,...,n-2,n}_{1,...,n-1}\big(T(\ell^{(n-1)}_j)\big)\xi_{\Lambda}\\
     =&(-1)^{1+n+j}\sqrt{-\frac{\prod_{s}(l_{n-1,j}-l_{n-2,s})\prod_{k\neq j}(l_{n-1,j}-l_{n-1,k})\prod_{i}(l_{n-1,j}-1-l_{n i})}{\prod_{k\neq j}(l_{n-1,j}-1-l_{n-1,k})}} \xi_{\Lambda-\delta_{n-1, j}}.
\end{align*}
Here the index $i$, $k$, $s$ in the product operator $\prod_{i}$, $\prod_{k\neq j}$, $\prod_{s}$, range from $1$ to $n$, $n-1$, $n-2$, and $j$ is not included in $\prod_{k\neq j}$. The number $l_{n i}=l_{n i}\left(\Lambda\right)$, $l_{n-1,j}=l_{n-1,j}\left(\Lambda\right)$, $l_{n-1,k}=l_{n-1,k}\left(\Lambda\right)$, $l_{n-2,s}=l_{n-2,s}\left(\Lambda\right)$ are all evaluated at the pattern $\Lambda$.
\end{lem}
\begin{proof}
We suppose that
\begin{equation*}
\Delta^{1,...,n-1}_{1,...,n-1}\left(T\left(x\right)\right)=(-1)^{n-1}c_{n-1} x^{n-1}+(-1)^{n-2}c_{n-2} x^{n-2}+\cdots -c_1 x+c_0.    
\end{equation*}
Here $c_{n-1}=1$ and $c_{n-1-i} =\sum_{1 \leq s_1< \cdots < s_i \leq n-1} \ell^{\left(n-1\right)}_{ s_1}\ell^{\left(n-1\right)}_{ s_2}\cdots \ell^{\left(n-1\right)}_{ s_i}$. By Proposition \ref{2}, we have
\begin{align}
    \left[\Delta^{1,...,n-1}_{1,...,n-1}\left(T\left(x\right)\right),e_{n-1,n}\right]&=\Delta^{1,...,n-2,n-1}_{1,...,n-2,n}\left(T\left(x\right)\right)\label{Eleft commutator},\\
    \left[e_{n,n-1},\Delta^{1,...,n-1}_{1,...,n-1}\left(T\left(x\right)\right)\right]&=\Delta^{1,...,n-2,n}_{1,...,n-2,n-1}\left(T\left(x\right)\right)\label{Eright commutator}.
\end{align}
As a result of \eqref{Eleft commutator} and \eqref{Eright commutator}, we have
\begin{align*}
    &\left(\Delta_{l}\right)^{1,...,n-2,n-1}_{1,...,n-2,n}\big(T(\ell^{(n-1)}_j)\big)=\big(-\ell^{(n-1)}_j\big)^{n-1}\left[c_{n-1},e_{n-1,n}\right]+\big(-\ell^{(n-1)}_j\big)^{n-2}\left[c_{n-2},e_{n-1,n}\right]+\cdots+\left[c_0,e_{n-1,n}\right],\\
    &\left(\Delta_{r}\right)^{1,...,n-2,n}_{1,...,n-2,n-1}\big(T(\ell^{(n-1)}_j)\big)=\left[e_{n,n-1},c_{n-1}\right]\big(-\ell^{(n-1)}_j\big)^{n-1}+\left[e_{n,n-1},c_{n-2}\right]\big(-\ell^{(n-1)}_j\big)^{n-2}+\cdots+\left[e_{n,n-1},c_0\right].
\end{align*}
From Proposition \ref{3}, we have
\begin{align*}
    &\sum_{i=0}^{n-1}\big(-\ell^{(n-1)}_j\big)^{i}c_{i}e_{n-1,n}\xi_{\Lambda}=\sum_{i=0}^{n-1}e_{n,n-1}c_{i}\big(-\ell^{(n-1)}_j\big)^{i}\xi_{\Lambda}=0,
    \\
    &\sum_{i=1}^{n-1}\big(-\ell^{(n-1)}_j\big)^{i}e_{n-1,n}c_{i}\xi_{\Lambda}=
    \left(-1\right)^{n+j}\sqrt{-\frac{\prod_{i} (l_{n-1,j}-l_{n i})\prod_{k\neq j}(l_{n-1,j}+1-l_{n-1,k})\prod_{s}(l_{n-1,j}+1-l_{n-2,s})}{\prod_{k\neq j}(l_{n-1,j}-l_{n-1,k})}}\xi_{\Lambda+\delta_{n-1,j}},
    \\
    &\sum_{i=0}^{n-1}c_{i}e_{n,n-1}\big(-\ell^{(n-1)}_j\big)^{i}\xi_{\Lambda}=\left(-1\right)^{n+j}\sqrt{-\frac{\prod_{s}(l_{n-1,j}-l_{n-2,s})\prod_{k\neq j}(l_{n-1,j}-l_{n-1,k})\prod_{i}(l_{n-1,j}-1-l_{n i})}{\prod_{k\neq j}(l_{n-1,j}-1-l_{n-1,k})}} \xi_{\Lambda-\delta_{n-1, j}}.
\end{align*}
The actions of $\left(\Delta_{l}\right)^{1,...,n-2,n-1}_{1,...,n-2,n}\big(T(\ell^{(n-1)}_j)\big)$, $\left(\Delta_r\right)^{1,...,n-2,n}_{1,...,n-2,n-1}\big(T(\ell^{(n-1)}_j)\big)$ are derived from this three identities.
\end{proof}
\subsection{Projection operator $\mathrm{Pr}_{i}$}
Let us introduce the projection operator $\mathrm{Pr}_{i}$ ( see e.g. \cite{MG,Molev}) here.
\begin{definition}\label{def Pr}
    For $1\leq i \leq n-1$, define $\mathrm{Pr}_{i}\in {\rm End}\left(L\left(\lambda\right)\right)\otimes {\rm End}\left(\mathbb{C}^{n-1}\right)$ as follows
    \begin{equation*}
    \mathrm{Pr}_{i}=\prod_{l\ne i}\frac{\left(T^{\left(n-1\right)}-\left(\ell^{\left(n-1\right)}_l+n-2\right)\cdot \mathrm{Id}_{n-1}\right)}{\left(\ell^{\left(n-1\right)}_i-\ell^{\left(n-1\right)}_l\right)}.
\end{equation*}
Here ${\rm Id}_{n-1}$ is the rank $n-1$ diagonal matrix with diagonal elements $1\in U\left({\mathfrak {gl}_n}\right)$.
\end{definition}
Since $T^{\left(n-1\right)}$ satisfies the following identity
\begin{equation*}
    \prod_{i=1}^{n}\left(T^{\left(n-1\right)}-\left(\ell^{\left(n-1\right)}_i+n-2\right)\cdot \mathrm{Id}_{n-1}\right)=0,
\end{equation*}
we have
\begin{align}
    &\mathrm{Pr}_{i}\cdot \mathrm{Pr}_{i}=\mathrm{Pr}_{i},\quad 1\leq i \leq n-1,\label{Pi Pi}\\
    &\mathrm{Pr}_{i} \cdot \mathrm{Pr}_{j}=0, \quad 1\leq i \neq j \leq n-1\label{Pi Pj}.
\end{align}
We can decompose $T$ as follows
\begin{equation}\label{sum P}
    T^{\left(n-1\right)}=\sum_{i=1}^{n-1} \left(\ell^{\left(n-1\right)}_{i}+n-2\right)\mathrm{Pr}_{i}.
\end{equation}
From \eqref{sum P}, we see that for rather general rational fucntions $f$, we have
\begin{equation*}
    f\left(T^{\left(n-1\right)}\right)=\sum_{i=1}^{n-1}f\left(\ell^{\left(n-1\right)}_i+n-2\right)\mathrm{Pr}_{i}.
\end{equation*}
Especially, if we choose $f=1$, we find that
\begin{equation}\label{sum Pi=Id}
    \sum_{i=1}^{n-1}\mathrm{Pr}_{i}=\mathrm{Id}_{n-1}.
\end{equation}
By using $\mathrm{Pr}_{i}$, we can prove the following lemma that will be used later
\begin{lem}\label{quantum square adj}
 Let $\Delta^{1,...,n-1}_{1,...,n-1}\left(T\left(x\right)\right)^{\prime}$ be the derivative of $\Delta^{1,...,n-1}_{1,...,n-1}\left(T\left(x\right)\right)$ with respect to $x$. For $1\leq k \leq n-1$, we have
\begin{align}\label{quantum chapoly}
\left(T^{\left(n-1\right)}-\ell_{k}^{\left(n-1\right)} \mathrm{Id}_{n-1}\right)^{* 2}=-\Delta^{1,...,n-1}_{1,...,n-1}\left(T\left(\ell^{\left(n-1\right)}_k\right)\right)^{\prime}\left(T^{\left(n-1\right)}-\ell_{k}^{\left(n-1\right)} \mathrm{Id}_{n-1} \right)^{*},
\end{align}
where $\left(T^{\left(n-1\right)}-\ell_{k}^{\left(n-1\right)} \mathrm{Id}_{n-1} \right)^{*}$ is adjoint matrix of $T^{(n-1)}-\ell_{k}^{(n-1)} {\rm Id}_{n-1}$, and $\left(T^{\left(n-1\right)}-\ell_{k}^{\left(n-1\right)}\mathrm{Id}_{n-1}\right)^{* 2}$ is the square of $\left(T^{\left(n-1\right)}-\ell_{k}^{\left(n-1\right)}\mathrm{Id}_{n-1}\right)^{*}$. 
\end{lem}
\begin{proof}

From \eqref{sum P} and $\left(T^{\left(n-1\right)}-x\right)^{*}=\Delta_{1,...,n-1}^{1,...,n-1}\left(T^{\left(n-1\right)}\left(x\right)\right) \left(T^{\left(n-1\right)}-(x+n-2)\right)^{-1}$, we have
\begin{equation*}
    \left(T^{\left(n-1\right)}-x\right)^{*}=\sum_{i=1}^{n-1}\frac{\prod_{l=1}^{n-1}\left(\ell^{\left(n-1\right)}_l-x\right)}{\ell^{\left(n-1\right)}_i-x}\mathrm{Pr}_{i},
\end{equation*}
\noindent and
\begin{align*}
\left(T^{\left(n-1\right)}-\ell^{\left(n-1\right)}_k\mathrm{Id}_{n-1}\right)^{*}=\prod_{l=1,l\neq k}^{n-1}\left(\ell_{l}^{\left(n-1\right)}-\ell_{k}^{\left(n-1\right)}\right)\mathrm{Pr}_{k}.
\end{align*}
From \eqref{Pi Pi}, we have
$
    \left(T^{\left(n-1\right)}-\ell_{k}^{\left(n-1\right)}\mathrm{Id}_{n-1} \right)^{* 2}=-\Delta^{1,...,n-1}_{1,...,n-1}\left(T\left(\ell^{\left(n-1\right)}_k\right)\right)^{\prime}\left(T^{\left(n-1\right)}-\ell_{k}^{\left(n-1\right)} \mathrm{Id}_{n-1}\right)^{*}.
$
\end{proof}

\subsection{Diagonalization of the upper left submatrix of $T$}\label{sec:diag}
Now let $\mathcal{P}_{h,n-1}$ and $\mathcal{Q}_{h,n-1}$ be matrices defined as follows
\begin{align*}
&\mathcal{P}_{h,n-1}=\left(\begin{array}{ccc}
\left(-1\right)^{1+1} \Delta_{\hat{1} ,..., n-1 }^{1 ,..., n-2 }\left(T\left(\ell^{\left(n-1\right)}_{1}\right)\right) & \cdots & \left(-1\right)^{1+n-1} \Delta_{\hat{1} ,..., n-1}^{1 ,..., n-2 }\left(T\left(\ell^{\left(n-1\right)}_{n-1}\right)\right) \\
\vdots &  \ddots & \vdots \\
\left(-1\right)^{n-1+1} \Delta_{1 ,..., \widehat{n-1} }^{1 ,..., n-2}\left(T\left(\ell^{\left(n-1\right)}_{1}\right)\right) & \cdots & \left(-1\right)^{n-1+n-1} \Delta_{1 ,..., \widehat{n-1} }^{1 ,..., n-2}\left(T\left(\ell^{\left(n-1\right)}_{n-1}\right)\right)
\end{array}\right),\\
&\mathcal{Q}_{h,n-1}=\left(\begin{array}{ccc}
\left(-1\right)^{1+1} \Delta_{1 ,..., n-2}^{\hat{1} ,..., n-1}\left(T\left(\ell^{\left(n-1\right)}_1\right)\right) & \cdots & \left(-1\right)^{1+n-1} \Delta^{1 ,..., \widehat{n-1} }_{1 ,..., n-2}\left(T\left(\ell^{\left(n-1\right)}_1\right)\right) \\
\vdots & \ddots & \vdots \\
\left(-1\right)^{n-1+1} \Delta_{1 ,..., n-2}^{\hat{1} ,..., n-1}\left(T\left(\ell^{\left(n-1\right)}_{n-1}\right)\right) & \cdots & \left(-1\right)^{n-1+n-1} \Delta^{1 ,..., \widehat{n-1}}_{1 ,..., n-2}\left(T\left(\ell^{\left(n-1\right)}_{n-1}\right)\right)
\end{array}\right).
\end{align*}
The quantum analog of Proposition \ref{PQ} is
\begin{proposition}\label{quantum PQ}
We have
\begin{align}\label{quantum PAAP}
T^{\left(n-1\right)} \mathcal{P}_{h,n-1}=\mathcal{P}_{h,n-1} T^{\left(n-1\right)}_{n-1}, \ \text{ and } \ 
   \mathcal{Q}_{h,n-1} T^{\left(n-1\right)}=T^{\left(n-1\right)}_{n-1}\mathcal{Q}_{h,n-1},
   \end{align}
   where $T^{\left(n-1\right)}$ is the upper left $\left(n-1\right)\times \left(n-1\right)$ submatrix of $T$, and 
   \begin{equation*}
       T^{\left(n-1\right)}_{n-1}=\mathrm{diag}\left(\ell^{\left(n-1\right)}_1+n-2,\ldots,\ell^{\left(n-1\right)}_{n-1}+n-2\right).
   \end{equation*} 
Furthermore, we have that 
\begin{equation}\label{qPQD}
\mathcal{Q}_{h,n-1}\mathcal{P}_{h,n-1}=D_{h,n-1},
\end{equation}
where $D_{h,n-1}={\rm diag}\left(\left(D_{h,n-1}\right)_1,....,\left(D_{h,n-1}\right)_{n-1}\right)$ is a diagonal matrix with the diagonal elements 
\begin{equation*}
\left(D_{h,n-1}\right)_i=
\prod_{l=1,l\ne i}^{n-1}\left(\ell^{\left(n-1\right)}_l-\ell^{\left(n-1\right)}_i\right)\prod_{l=1}^{n-2}\left(\ell^{\left(n-2\right)}_l-\ell^{\left(n-1\right)}_i\right), \text{ for } i=1,...,n-1.
\end{equation*}
\end{proposition}
\begin{proof}
First, the identity \eqref{quantum PAAP} follows from the quantum Laplace expansion in Proposition \ref{Laplace expansion} and the identity \eqref{qPQD} follows from Lemma \ref{quantum square adj}. 
\end{proof}

Just as in the classical case, to simplify the equation, we can diagonalize the upper left $\left(n-1\right)$-th submatrix of its coefficient matrix. 
Recall that in the classical case, the diagonalization is only taken on the open dense subset with the condition \eqref{nonres} and the conditions \eqref{ineq1}, \eqref{ineq2}. In the quantum case, for $h$ a real number the equation is nonresonant, and the analog of the  the conditions \eqref{ineq1}, \eqref{ineq2} is replaced by a subspace of the representation $L\left(\lambda\right)$ spanned by the basis elements $\xi_\Lambda$ such that the following inequalities are strict, i.e., 
\[L\left(\lambda\right)_0:={\rm span}\{\xi_\Lambda~|~\lambda_j^{\left(n-1\right)}\left(\Lambda\right)-\lambda^{\left(n-2\right)}_i\left(\Lambda\right)\ne 0 \}\subset L\left(\lambda\right).\]
However, just as in the classical setting, the Stokes matrices have no singularities. It is only the diagonalization matrix that introduces the singularities. Therefore, to compute the quantum Stokes matrices of equation \eqref{introqeq}, it is enough to do the diagonalization formally. Thus for simplicity, in below we do the computation in a formal setting.
\begin{rmk}
    The decomposition \eqref{sum P} of $T^{\left(n-1\right)}$ plays the same role as diagonalization. In fact, the relationship between $\mathcal{P}_{h,n-1}$, $\mathcal{Q}_{h,n-1}$ and $\mathrm{Pr}_{k}$ is clear. From \eqref{Pi Pi}, \eqref{Pi Pj} and \eqref{sum P}, we have
    \begin{align*}
        &T^{\left(n-1\right)} \cdot \mathrm{Pr}_{k}=\left(\ell^{\left(n-1\right)}_k+n-2\right)\cdot \mathrm{Pr}_{k},\\
        &\mathrm{Pr}_{k}\cdot  T^{\left(n-1\right)}=\left(\ell^{\left(n-1\right)}_k+n-2\right)\cdot \mathrm{Pr}_{k}.
    \end{align*}
    So every column (row) of $\mathrm{Pr}_k$ is column (row) eigenvector with eigenvalue $\ell^{\left(n-1\right)}_k+n-2$ and from Proposition \ref{quantum PQ} and \eqref{Pi Pi}, $\mathrm{Pr}_k$ has the following form
    \begin{equation}\label{Pk=PQ}
    \mathrm{Pr}_{k}=\left(\left(\mathcal{P}_{h,n-1}\right)_{1k},...,\left(\mathcal{P}_{h,n-1}\right)_{n-1,k}\right)^{\intercal}\cdot \frac{1}{\left(D_{h,n-1}\right)_k} \cdot \left(\left(\mathcal{Q}_{h,n-1}\right)_{k1},...,\left(\mathcal{Q}_{h,n-1}\right)_{k,n-1}\right).
\end{equation}
    As a result, from \eqref{sum Pi=Id} and \eqref{Pk=PQ}, we have
    \begin{equation*}
        \mathcal{P}_{h,n-1}\cdot D_{h,n-1}^{-1} \cdot \mathcal{Q}_{h,n-1}=\sum_{i=1}^{n-1}\mathrm{Pr}_{i}=\mathrm{Id}_{n-1}.
    \end{equation*}
     For the consistence with the classical case, we use the technique of diagonalization from now on.
\end{rmk}

Let us introduce $P_{h,n-1}$ and $Q_{h,n-1}$ (ignoring the pole issue in the denominator), 
\begin{equation*}
    P_{h,n-1}=\mathcal{P}_{h,n-1}\cdot \frac{1}{\sqrt{D_{h,n-1}}}, \ \ \ Q_{h,n-1}=\frac{1}{\sqrt{D_{h,n-1}}}\cdot \mathcal{Q}_{h,n-1}.
\end{equation*}
Then $Q_{h,n-1} P_{h,n-1}$ is identity when restricted to the subspace $L\left(\lambda\right)_0$. And we can introduce the matrix $T_{n-1}$
    \begin{equation}\label{Tn-1}
T_{n-1}:=\operatorname{diag}\left({P}_{h,n-1}^{-1},1\right)\cdot T \cdot \operatorname{diag}\left({P}_{h,n-1},1\right) =\left(\begin{array}{cccc}
\ell^{\left(n-1\right)}_1+n-2 & & 0 & \alpha^{\left(n-1\right)}_1 \\
 & \ddots & & \vdots \\
0 & & \ell^{\left(n-1\right)}_{n-1}+n-2 & \alpha^{\left(n-1\right)}_{n-1} \\
\beta^{\left(n-1\right)}_1 & \cdots & \beta^{\left(n-1\right)}_{n-1} & e_{n n} 
\end{array} \right),
\end{equation}
with entries given by
\begin{align*}
    &\alpha^{\left(n-1\right)}_k=\frac{(-1)^{k+n-1}}{\sqrt{\prod_{l=1,l\ne k}^{n-1}\left(\ell^{\left(n-1\right)}_l-\ell^{\left(n-1\right)}_k\right) \prod_{l=1}^{n-2}\left(\ell^{\left(n-2\right)}_{l}-\ell^{\left(n-1\right)}_k\right)}}\cdot \left(\Delta_{l}\right)_{1, ..., n-2, n}^{1, ..., n-2, n-1}\left(T\left(\ell^{\left(n-1\right)}_k\right)\right),\\
    &\beta^{\left(n-1\right)}_k=(-1)^{k+n-1}{\left(\Delta_{r}\right)_{1,..., n-2,n-1}^{1,..., n-2,n}\left(T\left(\ell^{\left(n-1\right)}_k\right)\right)}\cdot\frac{1}{\sqrt{\prod_{l=1,l\ne k}^{n-1}\left(\ell^{\left(n-1\right)}_l-\ell^{\left(n-1\right)}_k\right) \prod_{l=1}^{n-2}\left(\ell^{\left(n-2\right)}_{l}-\ell^{\left(n-1\right)}_k\right)}}.
\end{align*}
Taking the following partial fraction decomposition of the rational function
\begin{align*}
    &\frac{\prod_{i=1}^n\left(x-\frac{h}{2\pi \mathrm{i} }l_{n i}\left(\Lambda\right)\right)}{\prod_{i=1}^{n-1}\left(x-\frac{h}{2\pi \mathrm{i} }l_{n-1,i}\left(\Lambda\right)\right)}-x+\frac{h}{2\pi \mathrm{i} }\left(\sum_{i=1}^n l_{n i}\left(\Lambda\right)-\sum_{i=1}^{n-1}l_{n-1,i}\left(\Lambda\right)\right)\\
  = &\sum_{k=1}^{n-1} \frac{1}{x-\frac{h}{2\pi\mathrm{i} }l_{n-1,k}\left(\Lambda\right)}\frac{\prod_{i=1}^{n}\left(\frac{h}{2\pi \mathrm{i} }l_{n-1, k}\left(\Lambda\right)-\frac{h}{2\pi \mathrm{i} }l_{ni }\left(\Lambda\right)\right)}{\prod_{i=1,i\neq k}^{n-1}\left(\frac{h}{2\pi \mathrm{i} }l_{n-1, k}\left(\Lambda\right)-\frac{h}{2\pi \mathrm{i} }l_{n-1, i}\left(\Lambda\right)\right)},
\end{align*}
and replacing $x$ by $m+\frac{h l_{n j}\left(\Lambda\right)}{2\pi \mathrm{i} }$, $m\in \mathbb{Z}_{>0}$, lead to the following lemma that will be used later
\begin{proposition}\label{quantum ciq}
When applying actions on $L\left(\lambda\right)_0$, we have
\begin{equation*}
    \frac{\prod_{i=1}^n\big(m+\frac{h \ell^{\left(n\right)}_j}{2\pi \mathrm{i} }-\frac{h \ell^{\left(n\right)}_i}{2\pi \mathrm{i} }\big)}{\prod_{i=1}^{n-1}\big(m+\frac{h \ell^{\left(n\right)}_j}{2\pi \mathrm{i} }-\frac{h \ell^{\left(n-1\right)}_i}{2\pi \mathrm{i} }\big)}=m+\frac{h \ell^{\left(n\right)}_j}{2\pi \mathrm{i} }-\frac{h\left(e_{n n}-\left(n-1\right)\right)}{2\pi \mathrm{i} }-\sum_{k=1}^{n-1}\frac{h \beta^{\left(n-1\right)}_k}{2\pi \mathrm{i} } \frac{1}{m+\frac{h \ell^{\left(n\right)}_j}{2\pi \mathrm{i} }-\frac{h}{2\pi \mathrm{i} }\big(\ell^{\left(n-1\right)}_k-1\big)}\frac{h \alpha^{\left(n-1\right)}_k}{2\pi \mathrm{i} }.
\end{equation*}
\end{proposition}
\begin{proof}
By Lemma \ref{4}, for any Gelfand-Tsetlin basis vector $\xi_{\Lambda}\in L\left(\lambda\right)_0$, we have
\begin{align*}
\alpha_{k}^{\left(n-1\right)} \xi_{\Lambda} &=\sqrt{-\frac{\prod_{i=1}^n\left(l_{n-1,k}\left(\Lambda\right)-l_{n i}\left(\Lambda\right)\right)}{\prod_{i=1,i\neq k}^{n-1}\left(l_{n-1,k}\left(\Lambda\right)-l_{n-1,i}\left(\Lambda\right)\right)}}\xi_{\Lambda+\delta_{n-1,k}},
\\
\beta_{k}^{\left(n-1\right)} \xi_{\Lambda} &=\sqrt{-\frac{\prod_{i=1}^n\left(l_{n-1,k}\left(\Lambda\right)-1-l_{n i}\left(\Lambda\right)\right)}{\prod_{i=1,i\neq k}^{n-1}\left(l_{n-1,k}\left(\Lambda\right)-1-l_{n-1,i}\left(\Lambda\right)\right)}} \xi_{\Lambda-\delta_{n-1, k}}.
\end{align*}
As a result, we have
\begin{align*}
    &\left(m+\frac{h}{2\pi \mathrm{i} }\ell^{\left(n\right)}_j-\frac{h}{2\pi \mathrm{i} }\left(e_{n n}-\left(n-1\right)\right)-\sum_{k=1}^{n-1}\frac{h \beta^{\left(n-1\right)}_k}{2\pi \mathrm{i} } \frac{1}{m+\frac{h}{2\pi \mathrm{i} }\ell^{\left(n\right)}_j-\frac{h}{2\pi \mathrm{i} }\left(\ell^{\left(n-1\right)}_k-1\right)}\frac{h \alpha^{\left(n-1\right)}_k}{2\pi \mathrm{i} }\right)\xi_{\Lambda}\\
    =&\left(m+\frac{hl_{n j}\left(\Lambda\right)}{2\pi \mathrm{i} }-\frac{h}{2\pi \mathrm{i} }\left(\sum_{i=1}^n l_{n i}\left(\Lambda\right)-\sum_{i=1}^{n-1}l_{n-1,i}\left(\Lambda\right)\right)\right)\xi_{\Lambda}\\
    +&\sum_{k=1}^{n-1}\frac{1}{m+\frac{h}{2\pi \mathrm{i} }l_{n j}\left(\Lambda\right)-\frac{h}{2\pi \mathrm{i} }l_{n-1, k}\left(\Lambda\right)}\frac{\prod_{i=1}^{n}\left(\frac{h}{2\pi \mathrm{i} }l_{ni }\left(\Lambda\right)-\frac{h}{2\pi \mathrm{i} }l_{n-1, k}\left(\Lambda\right)\right)}{\prod_{i=1,i\neq k}^{n-1}\left(\frac{h}{2\pi \mathrm{i} }l_{n-1, i}\left(\Lambda\right)-\frac{h}{2\pi \mathrm{i} }l_{n-1, k}\left(\Lambda\right)\right)}\xi_{\Lambda}\\
=&\frac{\prod_{i=1}^n\left(m+\frac{h l_{nj}\left(\Lambda\right)}{2\pi \mathrm{i} }-\frac{h l_{ni}\left(\Lambda\right)}{2\pi \mathrm{i} }\right)}{\prod_{i=1}^{n-1}\left(m+\frac{h l_{nj}\left(\Lambda\right)}{2\pi \mathrm{i} }-\frac{h l_{n-1,i}\left(\Lambda\right)}{2\pi \mathrm{i} }\right)}\xi_{\Lambda}.
\end{align*}
Since $\xi_{\Lambda}$ is a basis vector, so when applying actions on any other vector in $L\left(\lambda\right)_0$, we still have the equality.
\end{proof}

Then similar to Proposition \ref{10}, if $F_h\left(z;E_n,T_{n-1}\right)$ is a solution of 
\begin{equation}\label{qn-1sys}
\frac{\mathrm{d}F}{\mathrm{d}z}=h\left(\mathrm{i} E_n+\frac{1}{2\pi \mathrm{i} }\frac{T_{n-1}}{z}\right)\cdot F, 
\end{equation}
then $\operatorname{diag}\left({P}_{h,n-1},1\right) F_h\left(z;E_n,T_{n-1}\right)\operatorname{diag}\left({P}_{h,n-1},1\right)^{-1}$ is a solution of the equation \eqref{introqeq}.
Moreover, the quantum Stokes matrices of the equations \eqref{introqeq} and \eqref{qn-1sys} are related by
\begin{equation*}
S_{h\pm }\left(E_n, T\right)=\operatorname{diag}\left({P}_{h,n-1},1\right)S_{h\pm}\left(E_n, T_{n-1}\right)\operatorname{diag}\left({P}_{h,n-1},1\right)^{-1}.
\end{equation*}

As in the classical setting, see Section \ref{secextension}, the quantum Stokes matrices $S_{h\pm }\left(E_n, T\right)$ of \eqref{introqeq} have no singularities. It is only the diagonalization matrix $P_{h,n-1}$ that introduces the singularities. Therefore, to compute the quantum Stokes matrices $S_{h\pm}\left(E_n, T\right)$, it is enough to do the diagonalization and compute $S_{h\pm}\left(E_n, T_{n-1}\right)$ formally.

\subsection{Solutions of the equation \eqref{qn-1sys} and its quantum Stokes matrices}\label{sec:qkummer}
In the following, given diagonal matrices $L_1,...,L_{n-1}, R_1,...,R_{n-1}\in {\rm End}(L(\lambda))$ (under the Gelfand-Tsetlin basis),
\begin{align*}
    L_k&={\rm diag}\big((L_k)_{11},...,(L_k)_{m m}\big)\in{\rm End}(L(\lambda)), \text{ for } k=1,...,n-1,\\
     R_k&={\rm diag}\big((R_k)_{11},...,(R_k)_{m m}\big)\in{\rm End}(L(\lambda)), \text{ for } k=1,...,n-1,
\end{align*}
and $m=\dim L(\lambda)$, we define
\begin{equation}\label{matrixgamma}
\Gamma(L_k):={\rm diag}\left(\Gamma((L_k)_{11}),...,\Gamma((L_k)_{mm})\right)\in {\rm End}(L(\lambda)),
\end{equation}
and define ${ }_{n-1} F_{n-1}\left(L_1, ...,L_{n-1}; R_1,..., R_{n-1} ; z
\right)\in {\rm End}(L(\lambda))$ as the diagonal matrix whose $(i,i)$-entry is
\[{ }_{n-1} F_{n-1}\left((L_1)_{ii}, ...,(L_{n-1})_{ii}; (R_1)_{ii},..., (R_{n-1})_{ii} ; z
\right), \text{ for } i=1,...,m.\]
\begin{proposition}\label{qsolthm}
The equation \eqref{qn-1sys}
has a solution 
\begin{equation}\label{quantum diagFz}
    F_h\left(z;E_n,T_{n-1}\right)=Y_h \cdot H_h\left(z\right) \cdot z^{\frac{h}{2\pi \mathrm{i} }T_{n}},
\end{equation}
where $Y_h=\operatorname{diag}\left(\alpha_{1}^{\left(n-1\right)}, ..., \alpha_{n-1}^{\left(n-1\right)}, 1\right)$, $T_{n}=\operatorname{diag}\left(\ell^{\left(n\right)}_1+n-1,\ell^{\left(n\right)}_2+n-1,...,\ell^{\left(n\right)}_{n}+n-1\right)$,
and the matrix $H_h\left(z\right)$ is an $n\times n$ matrix with entries as follows, for $1 \leq i \leq n-1$ and $1 \leq j \leq n$,
\begin{align*}
&H_h\left(z\right)_{i j}=\frac{1}{\ell^{\left(n\right)}_j-\ell^{\left(n-1\right)}_i} \cdot{ }_{n-1} F_{n-1}\left(\alpha_{i j, 1}, ..., \alpha_{i j, n-1}; \beta_{i j, 1}, ..., \widehat{\beta_{i j, j}}, ..., \beta_{i j, n+1} ;h\cdot \mathrm{i} \cdot z
\right),\\
	&H_h\left(z\right)_{n j}= { }_{n-1} F_{n-1}\left(\alpha_{n j, 1}, ..., \alpha_{n j, n-1}; \beta_{n j, 1}, ..., \widehat{\beta_{n j, j}}, ..., \beta_{n j, n} ;h\cdot \mathrm{i} \cdot z\right),
\end{align*}
with the variables $\left\{\alpha_{i j, l}\right\}$ and $\left\{\beta_{i j, l}\right\}$ given by
\begin{align*}
&\alpha_{i j, i} =\frac{h}{2 \pi \mathrm{i} }\left( \ell_{j}^{\left(n\right)}-\ell_{i}^{\left(n-1\right)}\right), \quad 1\leq i \leq n-1, \, 1\leq j \leq n,\\
&\alpha_{i j, l}=1+\frac{h}{2 \pi \mathrm{i} }\left( \ell_{j}^{\left(n\right)}-\ell_{l}^{\left(n-1\right)}\right),\quad 1 \leq i, j \leq n,\, 1 \leq l \leq n-1, \, l \neq i, \\
&\beta_{i j, l}=1+\frac{h}{2 \pi \mathrm{i} }\left(\ell_{j}^{\left(n\right)}-\ell_{l}^{\left(n\right)}\right), \quad 1 \leq i, j,l \leq n.
\end{align*}
The notation $\widehat{\beta_{i j, j}}$ means the term $\beta_{i j, j}$ is skipped, for $1\leq i \leq n$.
\end{proposition}
Before the proof of Proposition \ref{qsolthm}, we introduce the following commutation relations that will be used later
\begin{lem}\label{commu rels}
    For $1\leq i,j \leq n-1$, we have two relations
\begin{align*}
\left[\ell^{\left(n-1\right)}_j,\alpha^{\left(n-1\right)}_i\right]&=\delta_{ij}\alpha^{\left(n-1\right)}_j,\\
    \left[\ell^{\left(n-1\right)}_j,\beta^{\left(n-1\right)}_i\right]&=-\delta_{ij}\beta^{\left(n-1\right)}_j.
\end{align*}
\end{lem}

\begin{proof}
    From Lemma \ref{4}, we find that for any Gelfand-Tsetlin basis vector $\xi_{\Lambda}\in L\left(\Lambda\right)$, we have
\begin{align*}
    &\left[\ell^{\left(n-1\right)}_j,\left(\Delta_{l}\right)_{1, ..., n-2, n}^{1, ..., n-2, n-1}\left(T\left(\ell^{\left(n-1\right)}_i\right)\right)\right]\xi_{\Lambda}=\delta_{ij}\left(\Delta_{l}\right)_{1, ..., n-2, n}^{1, ..., n-2, n-1}\left(T\left(\ell^{\left(n-1\right)}_i\right)\right)\xi_{\Lambda},\\
    &\left[\ell^{\left(n-1\right)}_j,{\left(\Delta_{r}\right)_{1,..., n-2,n-1}^{1,..., n-2,n}\left(T\left(\ell^{\left(n-1\right)}_i\right)\right)}\right]\xi_{\Lambda}=-\delta_{ij}{\left(\Delta_{r}\right)_{1,..., n-2,n-1}^{1,..., n-2,n}\left(T\left(\ell^{\left(n-1\right)}_i\right)\right)}\xi_{\Lambda}.
\end{align*}
Recall that 
\begin{align*}
    &\alpha^{\left(n-1\right)}_i=\frac{(-1)^{i+n-1}}{\sqrt{\left(D_{h,n-1}\right)_i}}\left(\Delta_{l}\right)_{1, ..., n-2, n}^{1, ..., n-2, n-1}\left(T\left(\ell^{\left(n-1\right)}_i\right)\right),\\
    &\beta^{\left(n-1\right)}_i=(-1)^{i+n-1}{\left(\Delta_{r}\right)_{1,..., n-2,n-1}^{1,..., n-2,n}\left(T\left(\ell^{\left(n-1\right)}_i\right)\right)}\frac{1}{\sqrt{\left(D_{h,n-1}\right)_i}},
\end{align*} 
and $\ell^{\left(n-1\right)}_j$ commutes with $\frac{1}{\sqrt{\left(D_{h,n-1}\right)_i}}$, so we have the commutation relations in Lemma \ref{commu rels}.
\end{proof}
\begin{proof}[Proof of Proposition \ref{qsolthm}]
The series expansion of the matrix function \eqref{quantum diagFz} is \begin{equation}\label{quantum seriesexp}
    F_h\left(z;E_n,T_{n-1}\right)=\sum_{m=0}^{\infty} L_{h,m} z^{m} z^{\frac{h}{2 \pi \mathrm{i} } T_{n}},
\end{equation}
where $L_{h,m}$ is the coefficient of the expansion of $Y_h H_h\left(z\right)=\sum L_{h,m} z^m$. Following Definition \ref{ghf}, the 
$n\times n$ matrix $L_{h,m}$ has entries
        \begin{align}\label{qLM1}
            &\left(L_{h,m}\right)_{i j}=\alpha^{\left(n-1\right)}_i \frac{h^m\mathrm{i} ^m}{m!}\frac{1}{\left(\ell^{\left(n\right)}_j-\ell^{\left(n-1\right)}_i\right)}\frac{\prod_{l=1}^{n-1}\left(\alpha_{i j,l}\right)_m}{\prod_{l=1,l\ne j}^{n}\left(\beta_{ij,l}\right)_m}, \ \ 1\leq i \leq n-1, \ \  1\leq j \leq n,\\ \label{qLM2}
            &\left(L_{h,m}\right)_{nj}=\frac{h^m \mathrm{i} ^m}{m!}\frac{\prod_{l=1}^{n-1}\left(\alpha_{n j,l}\right)_m}{\prod_{l=1,l\ne j}^{n}\left(\beta_{nj,l}\right)_m},
        \end{align}
Let us verify that \eqref{quantum seriesexp}
is indeed a series solution of the system \eqref{qn-1sys}. Plugging \eqref{quantum seriesexp}
into \eqref{qn-1sys} and comparing the coefficients, we only need to verify the recursion relation
\begin{align}\label{qL0 requirement}
L_{h,0} T_{n}&=T_{n-1} L_{h,0}, \\
\label{qrecur}\allowdisplaybreaks
    m L_{h,m}+\frac{h}{2\pi\mathrm{i} }L_{h,m}T_{n}&=h\cdot \mathrm{i} \cdot E_n L_{h,m-1}+\frac{h}{2\pi \mathrm{i} } T_{n-1} L_{h,m}, \ \ m\ge 1.
\end{align}
The matrix identity \eqref{qL0 requirement} can be verified directly via Lemma \ref{commu rels} and the explicit expression of
\begin{align*}
    L_{h,0}=\left(\begin{array}{ccc}
 \alpha_{1}^{\left(n-1\right)}\frac{1}{\ell_{1}^{\left(n\right)}-\ell_{1}^{\left(n-1\right)}} & \cdots &  \alpha_{1}^{\left(n-1\right)}\frac{1}{\ell_{n}^{\left(n\right)}-\ell_{1}^{\left(n-1\right)}} \\
\vdots & & \vdots \\
 \alpha_{n-1}^{\left(n-1\right)}\frac{1}{\ell_{1}^{\left(n\right)}-\ell_{n-1}^{\left(n-1\right)}} & \cdots & \alpha_{n-1}^{\left(n-1\right)}\frac{1 }{\ell_{n}^{\left(n\right)}-\ell_{n-1}^{\left(n-1\right)}} \\
1 & \cdots & 1
\end{array}\right).
\end{align*}
On the one hand, the first $n-1$ rows of the matrix identity \eqref{qrecur}, for $1\leq i\leq n-1,\, 1\leq j \leq n$, are just
\begin{align}\label{qfirst n-1 row}
    \left(m+\frac{h}{2\pi \mathrm{i} }\left(\ell^{\left(n\right)}_j+n-1\right)\right)\left(L_{h,m}\right)_{i j}=\frac{h}{2\pi \mathrm{i} }\left(\ell^{\left(n-1\right)}_i+n-2\right) \left(L_{h,m}\right)_{i j}+\frac{h}{2\pi \mathrm{i} }\alpha^{\left(n-1\right)}_i \left(L_{h,m}\right)_{n j}.
\end{align}
Replacing $\left(L_{h,m}\right)_{ij}$, $\left(L_{h,m}\right)_{nj}$ by \eqref{qLM1}, \eqref{qLM2}, one checks that \eqref{qfirst n-1 row} is indeed an identity. The last row of the matrix identity \eqref{qrecur} is 
\begin{align}\label{qlast row}
    \left(m+\frac{h}{2\pi \mathrm{i} }\left(\ell^{\left(n\right)}_j+n-1\right)\right)\left(L_{h,m}\right)_{n j}=h\cdot \mathrm{i} \cdot \left(L_{h,m-1}\right)_{n j}+\frac{h}{2\pi \mathrm{i} }\sum_{k=1}^{n-1}\beta^{\left(n-1\right)}_k \left(L_{h,m}\right)_{k j}+\frac{h}{2\pi \mathrm{i} } e_{n n}\left(L_{h,m}\right)_{n j}.
\end{align}
To verify \eqref{qlast row}, by
substituting \eqref{qfirst n-1 row} into \eqref{qlast row}, we should prove that for $1\leq j \leq n$,
\begin{align*}
    \left(m+\frac{h\left(\ell^{\left(n\right)}_j+n-1\right)}{2\pi \mathrm{i} }-\sum_{k=1}^{n-1}\frac{h \beta^{\left(n-1\right)}_k}{2\pi \mathrm{i} } \frac{1}{m+\frac{h}{2\pi \mathrm{i} }\left(\ell^{\left(n\right)}_j-\ell^{\left(n-1\right)}_k+1\right)}\frac{h \alpha^{\left(n-1\right)}_k}{2\pi \mathrm{i} }-\frac{he_{n n}}{2\pi \mathrm{i} }\right) \left(L_{h,m}\right)_{n j}= h \mathrm{i}   \left(L_{h,m-1}\right)_{n j},
\end{align*}
and this can be verified directly in a similar way as \eqref{substitute first n-1 rows into the last row} by Proposition \ref{quantum ciq}. 

\end{proof}

Similar to Proposition \ref{qsol}, there exist solutions $F_{h,0}\left(z;E_n,T_{n-1}\right)$ and $F_{h,-1}\left(z;E_n,T_{n-1}\right)$ of the system \eqref{qn-1sys} such that for $i=0,-1$
\begin{equation*}
F_{h,i}\left(z;E_n,T_{n-1}\right)\cdot \mathrm{e}^{-{h\mathrm{i} E_nz}}z^{-\frac{h \delta_{n-1}\left(T_{n-1}\right)}{2\pi\mathrm{i}}} \sim \mathrm{Id},\quad z \rightarrow \infty \text{ in $\operatorname{Sect}_i$},
\end{equation*}
where
\[\delta_{n-1}\left(T_{n-1}\right)=\operatorname{diag}\left({\ell}^{\left(n-1\right)}_1+n-2,{\ell}^{\left(n-1\right)}_2+n-2,\cdots,{\ell}^{\left(n-1\right)}_{n-1}+n-2, e_{nn}\right).\] 
\begin{proposition}\label{6}
The transition matrices $U_{h,0}$ and $U_{h,-1}$, determined by
\begin{align}
&F_{h}\left(z;E_n,T_{n-1}\right) =F_{h,0}\left(z;E_n,T_{n-1}\right) \cdot Y_h U_{h,0},\quad z\in \operatorname{Sect}_{0},\label{qU0}\\
&F_{h}\left(z;E_n,T_{n-1}\right)=F_{h,-1}\left(z;E_n,T_{n-1}\right) \cdot Y_h U_{h,-1},\quad z\in \operatorname{Sect}_{-1}\label{qU1},
\end{align}
are given as follows, for $1\leq i \leq n-1$, and $1\leq j \leq n$,
\begin{align*}
    &\left(U_{h,0}\right)_{i j}=\frac{1}{\ell^{\left(n\right)}_j-\ell^{\left(n-1\right)}_i}\frac{\prod_{l=1,l\ne j}^{n} \Gamma\left(\beta_{ij,l}\right)}{\prod_{l=1,l\ne i}^{n-1} \Gamma\left(\alpha_{ij,l}\right)}\frac{\prod_{l=1,l\ne i}^{n-1} \Gamma\left(\alpha_{ij,l}-\alpha_{ij,i}\right)}{\prod_{l=1,l\ne j}^{n}\Gamma\left(\beta_{ij,l}-\alpha_{ij,i}\right)}\frac{1}{\left(h\mathrm{e}^{-\frac{\pi \mathrm{i} }{2}}\right)^{\alpha_{ij,i}}},\\
&\left(U_{h,-1}\right)_{i j}=\frac{1}{\ell^{\left(n\right)}_j-\ell^{\left(n-1\right)}_i}\frac{\prod_{l=1,l\ne j}^{n} \Gamma\left(\beta_{ij,l}\right)}{\prod_{l=1,l\ne i}^{n-1} \Gamma\left(\alpha_{ij,l}\right)}\frac{\prod_{l=1,l\ne i}^{n-1} \Gamma\left(\alpha_{ij,l}-\alpha_{ij,i}\right)}{\prod_{l=1,l\ne j}^{n}\Gamma\left(\beta_{ij,l}-\alpha_{ij,i}\right)}\frac{1}{\left(h\mathrm{e}^{\frac{3\pi \mathrm{i} }{2}}\right)^{\alpha_{ij,i}}},
\end{align*}
and for $1\leq j\leq n$,
\begin{align*}
\left(U_{h,0}\right)_{nj}=\left(U_{h,-1}\right)_{nj}=\frac{\prod_{l=1,l \ne j}^{n} \Gamma\left(\beta_{nj,l}\right)}{\prod_{l=1}^{n-1}\Gamma\left(\alpha_{nj,l}\right)}\frac{1}{\left(h\mathrm{e}^{\frac{\pi \mathrm{i} }{2}}\right)^{\alpha_{nj,n}}},
\end{align*}
where for convenience, we denote
\begin{align*}
&\ell^{\left(n-1\right)}_n:=e_{n n}-\left(n-1\right)
=\sum_{l=1}^{n}\ell^{\left(n\right)}_l-\sum_{l=1}^{n-1}\ell^{\left(n-1\right)}_l,\\
    &\alpha_{n j,n}:=\sum_{l=1,l\ne j}^{n}\beta_{nj,l}-\sum_{l=1}^{n-1}\alpha_{nj,l}=\frac{h}{2\pi \mathrm{i} }\left(\ell^{\left(n\right)}_j-\ell^{\left(n-1\right)}_n\right),
    \quad  1\leq j \leq n.
\end{align*}
\end{proposition}

\begin{proof}[Proof of Proposition \ref{6}]
Let $\hat{F}_h\left(z;E_n,T_{n-1}\right)$ be the formal solution of the form
\begin{eqnarray}\label{qn-1syst formalsum}
\hat{F}_h\left(z;E_n,T_{n-1}\right)=\hat{K}_h\left(z\right) e^{{h\mathrm{i} E_nz}}z^{\frac{h\delta_{n-1}\left(T_{n-1}\right)}{2\pi\mathrm{i}}}, \ \ \ {\it for} \ \hat{K}_h=\mathrm{Id}+K_{h,1}z^{-1}+K_{h,2}z^{-2}+\cdot\cdot\cdot. \end{eqnarray}
    From the asymptotics of $F_{h,0}\left(z;E_n,T_{n-1}\right)$ within ${\rm Sect}_0$ and the identity \eqref{qU0}, we have
\begin{equation}\label{Fh Uinverse}
    F_{h}\left(z;E_n,T_{n-1}\right) U_{h,0}^{-1} Y_h^{-1} \cdot e^{-{h\mathrm{i} E_nz}}z^{-\frac{h \delta_{n-1}\left(T_{n-1}\right)}{2\pi\mathrm{i}}} \sim \hat{K}_h\left(z\right),\quad z\rightarrow \infty \text{ within } \operatorname{Sect}_0.
\end{equation} 
Here for $1\le i\neq  k \le n-1,$ we have
\begin{align}\label{qF1}
    &\left(K_{h,1}\right)_{i k}=-\mathrm{i} \alpha^{\left(n-1\right)}_i\frac{\prod_{l=1}^{n}\frac{h}{2\pi \mathrm{i} }\left(\ell^{\left(n-1\right)}_k-\ell^{\left(n\right)}_l\right)}{h\prod^n_{l=1,\, l\ne k}\frac{h}{2\pi \mathrm{i} }\left(\ell^{\left(n-1\right)}_k-\ell^{\left(n-1\right)}_l\right)\left(\frac{h}{2\pi \mathrm{i} }\left(\ell^{\left(n-1\right)}_k-\ell^{\left(n-1\right)}_i\right)-1\right)}\left(\alpha^{\left(n-1\right)}_k\right)^{-1},\\ 
    \label{qF2}
    &\left(K_{h,1}\right)_{n k}=\frac{2\pi}{h^2}\frac{\prod^{n}_{l=1}\frac{h}{2\pi \mathrm{i} }\left(\ell^{\left(n-1\right)}_k-\ell^{\left(n\right)}_l\right)}{\prod^{n-1}_{l=1,l\ne k}\frac{h}{2\pi \mathrm{i} }\left(\ell^{\left(n-1\right)}_k-\ell^{\left(n-1\right)}_l\right)}\left(\alpha^{\left(n-1\right)}_k\right)^{-1},\\
    \label{qFn}
    &\left(K_{h,1}\right)_{i n}=-\frac{\alpha^{\left(n-1\right)}_i}{2\pi},
\end{align}
which follow from the recursion of the coefficients of formal fundamental solution
\begin{align*}
    \left[K_{h,m+1},h\mathrm{i} E_n\right]=\frac{h}{2\pi \mathrm{i} }T_{n-1}K_{h,m}+K_{h,m}\left(-\frac{h}{2\pi \mathrm{i} }\delta_{n-1}\left(T_{n-1}\right)+m\right).
\end{align*}
Similar to Proposition \ref{mainpro}, \eqref{Fh Uinverse} in turn implies, for $1\leq i \leq n-1$ and $1\leq j \leq n$, when $z\rightarrow \infty$ in $\operatorname{Sect}_0$,
\begin{equation}
\begin{aligned}\label{Fh entry asym1}
    F_h\left(z;E_n,T_{n-1}\right)_{i j}\sim &\sum_{k=1,k\ne i}^{n-1}z^{\frac{h}{2\pi \mathrm{i} }\left(\ell^{\left(n-1\right)}_k+n-1\right)-1}\left(K_{h,1}\right)_{i k} \alpha^{\left(n-1\right)}_k\left(U_{h,0}\right)_{k j}\\
    &+z^{\frac{h}{2\pi \mathrm{i} }\left(\ell^{\left(n-1\right)}_i+n-2\right)}\alpha^{\left(n-1\right)}_i\left(U_{h,0}\right)_{i j}\\
    &+z^{\frac{h}{2\pi \mathrm{i} }\left(e_{n n}+1\right)-1}\operatorname{e}^{h\mathrm{i} z}\left(K_{h,1}\right)_{in} \left(U_{h,0}\right)_{nj}
\end{aligned}
\end{equation}
and
\begin{align}
    F_h\left(z;E_n,T_{n-1}\right)_{n j}\sim \sum_{k=1}^{n-1}z^{\frac{h}{2\pi \mathrm{i} }\left(\ell^{\left(n-1\right)}_k+n-1\right)-1} \left(K_{h,1}\right)_{n k} \alpha^{\left(n-1\right)}_k\left(U_{h,0}\right)_{k j}+z^{\frac{h}{2\pi \mathrm{i} }e_{n n}}\operatorname{e}^{h\mathrm{i} z} \left(U_{h,0}\right)_{nj}\label{Fh entry asym2},
\end{align}
where we use the following commutation relations derived from Lemma \ref{commu rels} to move power functions to the front
\begin{align*}
    z^{\frac{h}{2\pi \mathrm{i} }\left(\ell^{\left(n-1\right)}_k+n-1\right)-1}\left(K_{h,1}\right)_{i k}&=\left(K_{h,1}\right)_{i k} z^{\frac{h}{2\pi \mathrm{i} }\left(\ell^{\left(n-1\right)}_k+n-2\right)-1},\\
    z^{\frac{h}{2\pi \mathrm{i} }\left(\ell^{\left(n-1\right)}_k+n-1\right)-1}\left(K_{h,1}\right)_{n k}&=\left(K_{h,1}\right)_{n k}z^{\frac{h}{2\pi \mathrm{i} }\left(\ell^{\left(n-1\right)}_k+n-2\right)-1},\\
    z^{\frac{h}{2\pi \mathrm{i} }\left(e_{n n}+1\right)-1}\left(K_{h,1}\right)_{in}&=\left(K_{h,1}\right)_{in}z^{\frac{h}{2\pi \mathrm{i} }e_{n n}-1}.
\end{align*}

To get the entries of $U_{h,0}$, we only need to study the asymptotics of the solution $F_{h}\left(z;E_n,T_{n-1}\right)$.
From the asymptotic behavior \eqref{asy 0 pi} in Lemma \ref{nfn}, 
 we get as $z\rightarrow\infty$ within $\mathrm{Sect}_0$, for $1\leq i \leq n-1$, $1\leq j \leq n$,
\begin{equation}
\begin{aligned}
F_h\left(z;E_n,T_{n-1}\right)_{i j} \sim& \sum_{k=1,k\neq i}^{n-1}z^{\frac{h}{2\pi \mathrm{i}}\left(\ell^{(n-1)}_k+n-1\right)-1}\alpha^{(n-1)}_i\frac{C_{i j,k}}{\left(\ell^{\left(n\right)}_j-\ell^{\left(n-1\right)}_i\right)\cdot \left(h\mathrm{e}^{-\frac{\pi \mathrm{i} }{2}}\right)^{\alpha_{ij,k}}} \\
&+z^{\frac{h}{2\pi \mathrm{i}}\left(\ell^{(n-1)}_i+n-2\right)}\alpha^{(n-1)}_i\frac{C_{i j,k}}{\left(\ell^{\left(n\right)}_j-\ell^{\left(n-1\right)}_i\right)\cdot \left(h\mathrm{e}^{-\frac{\pi \mathrm{i} }{2}}\right)^{\alpha_{ij,k}}}
\\
&+z^{\frac{h}{2\pi \mathrm{i}}\left(\ell^{(n-1)}_n+n\right)-1} \mathrm{e}^{h\cdot\mathrm{i} z} \alpha^{(n-1)}_i \frac{C_{i j,n}}{\left(\ell^{\left(n\right)}_j-\ell^{\left(n-1\right)}_i\right)\cdot \left(h\mathrm{e}^{\frac{\pi \mathrm{i} }{2}}\right)^{\alpha_{ij,n}}} , \label{qasym move1} 
\end{aligned}
\end{equation}
and
\begin{align}
    &F_h\left(z;E_n,T_{n-1}\right)_{n j} \sim  \sum_{k=1}^{n-1} z^{\frac{h}{2\pi \mathrm{i}}\left(\ell^{(n-1)}_k+n-1\right)-1} \frac{C_{n j,k}}{\left(h\mathrm{e}^{-\frac{\pi \mathrm{i} }{2}}\right)^{\alpha_{nj,k}}} +z^{\frac{h}{2\pi \mathrm{i}}\left(\ell^{(n-1)}_n+n-1\right)}\mathrm{e}^{h\cdot \mathrm{i} z} \frac{C_{n j,n}}{\left(h\mathrm{e}^{\frac{\pi \mathrm{i} }{2}}\right)^{\alpha_{nj,n}}} .\label{qasym move2}
\end{align}
where the constants $C_{i j,k}$, for $1\leq i,j \leq n$, are
\begin{align*}
    &C_{i j,k}=\frac{\prod_{l=1,l\ne j}^{n} \Gamma\left(\beta_{ij,l}\right)}{\prod_{l=1,l\ne k}^{n-1} \Gamma\left(\alpha_{ij,l}\right)}\frac{\prod_{l=1,l\ne k}^{n-1} \Gamma\left(\alpha_{ij,l}-\alpha_{ij,k}\right)}{\prod_{l=1,l\ne j}^{n}\Gamma\left(\beta_{ij,l}-\alpha_{ij,k}\right)}, \ \ 1\leq k \leq n-1, \\
    &C_{i j,n}=\frac{\prod_{l=1,l \ne j}^{n} \Gamma\left(\beta_{ij,l}\right)}{\prod_{l=1}^{n-1}\Gamma\left(\alpha_{ij,l}\right)},
\end{align*}
and $\alpha_{i j,n}$, for $1\leq i \leq n-1$, $1\leq j \leq n$, are
\begin{equation*}
    \alpha_{i j,n}:=\sum_{l=1,l\ne j}^{n}\beta_{ij,l}-\sum_{l=1}^{n-1}\alpha_{ij,l}
    =1+\frac{h}{2\pi \mathrm{i} }\left(\ell^{\left(n\right)}_j-\ell^{\left(n-1\right)}_n\right).
\end{equation*}
In \eqref{qasym move1}, the following commutation relations derived from Lemma \ref{commu rels} in used to move power functions to the front
\begin{align*}
    &z^{\frac{h}{2\pi \mathrm{i}}\left(\ell^{(n-1)}_i+n-2\right)}\alpha^{(n-1)}_i=\alpha^{(n-1)}_i z^{\frac{h}{2\pi \mathrm{i}}\left(\ell^{(n-1)}_i+n-1\right)},\\
    &z^{\frac{h}{2\pi \mathrm{i}}\left(\ell^{(n-1)}_n+n\right)-1}  \alpha^{(n-1)}_i=\alpha^{(n-1)}_i z^{\frac{h}{2\pi \mathrm{i}}\left(\ell^{(n-1)}_n+n-1\right)-1}.
\end{align*}
Comparing the coefficients in the asymptotic behaviors \eqref{Fh entry asym1}, \eqref{Fh entry asym2}, \eqref{qasym move1} and \eqref{qasym move2},
we can determine $U_{h,0}$ uniquely as follows. 
\begin{itemize}
    \item For $1\leq k \leq n-1$ and $1\leq j \leq n$, we find that $\left(U_{h,0}\right)_{k j}$ satisfies the following $n$ equations
\begin{align*}
    \alpha^{\left(n-1\right)}_i \frac{1}{\ell^{\left(n\right)}_j-\ell^{\left(n-1\right)}_i}C_{i j, k}\frac{1}{h^{\alpha_{ij,k}}}\mathrm{e}^{\frac{\pi \mathrm{i} }{2}\alpha_{ij,k}}&=\left(K_{h,1}\right)_{i k}  
    \alpha^{\left(n-1\right)}_k \left(U_{h,0}\right)_{k j}, \quad 1\leq i\neq k \leq n-1,\\
     \alpha^{\left(n-1\right)}_k \frac{1}{\ell^{\left(n\right)}_j-\ell^{\left(n-1\right)}_k}C_{k j, k}\mathrm{e}^{\frac{\pi \mathrm{i} }{2}\alpha_{kj,k}}&= \alpha^{\left(n-1\right)}_k \left(U_{h,0}\right)_{k j},\\
     C_{nj,k}\frac{1}{h^{\alpha_{nj,k}}}\mathrm{e}^{\frac{\pi \mathrm{i} }{2}\alpha_{nj,k}}& =\left(K_{h,1}\right)_{n k} \alpha^{\left(n-1\right)}_k \left(U_{h,0}\right)_{k j},
\end{align*}
which imply that (using identities \eqref{qF1} and \eqref{qF2})
\begin{align*}
    \left(U_{h,0}\right)_{k j}=\frac{1}{\ell^{\left(n\right)}_j-\ell^{\left(n-1\right)}_k}\frac{\prod_{l=1,l\ne j}^{n} \Gamma\left(\beta_{kj,l}\right)}{\prod_{l=1,l\ne k}^{n-1} \Gamma\left(\alpha_{kj,l}\right)}\frac{\prod_{l=1,l\ne k}^{n-1} \Gamma\left(\alpha_{kj,l}-\alpha_{kj,k}\right)}{\prod_{l=1,l\ne j}^{n}\Gamma\left(\beta_{kj,l}-\alpha_{kj,k}\right)}\frac{1}{\left(h\mathrm{e}^{-\frac{\pi \mathrm{i} }{2}}\right)^{\alpha_{kj,k}}}.
\end{align*}
\item For $1\leq j \leq n$, we find that $\left(U_{h,0}\right)_{n j}$ satisfies the following $n$ equations
\begin{align*}
    \alpha^{\left(n-1\right)}_i \frac{1}{\ell^{\left(n\right)}_j-\ell^{\left(n-1\right)}_i}C_{i j,n}\frac{1}{h^{\alpha_{ij,n}}} \mathrm{e}^{-\frac{\pi \mathrm{i} }{2}\alpha_{ij,n}}=\left(K_{h,1}\right)_{i n} \left(U_{h,0}\right)_{nj}, \quad 1\leq i \leq n-1,
\end{align*}
and
\begin{align*}
    C_{n j,n}\frac{1}{h^{\alpha_{n j,n}}}\mathrm{e}^{-\frac{\pi \mathrm{i} }{2}\alpha_{n j,n}}=\left(U_{h,0}\right)_{n j},
\end{align*}
which imply that (using the identity \eqref{qFn})
\begin{align*}
    \left(U_{h,0}\right)_{nj}=\frac{\prod_{l=1,l \ne j}^{n} \Gamma\left(\beta_{nj,l}\right)}{\prod_{l=1}^{n-1}\Gamma\left(\alpha_{nj,l}\right)}\frac{1}{\left(h\mathrm{e}^{\frac{\pi \mathrm{i} }{2}}\right)^{\alpha_{nj,n}}}.
\end{align*}
\end{itemize}
We can deduce $U_{h,-1}$ in the same manner as $U_{h,0}$ using \eqref{asy -pi 0} in Lemma \ref{nfn}.
\end{proof}
\begin{rmk}
    Here we note that the main difference between the proof of Proposition \ref{mainpro} and \ref{6} is that we need to consider the commutation relations of Lemma \ref{commu rels} in the proof of Proposition \ref{6}.
\end{rmk}
\begin{thm}\label{q9}
The Stokes matrix $S_{h+}\left(E_n,T_{n-1}\right)$ of the system \eqref{qn-1sys} is upper triangular with the form
\begin{equation}\label{explicit Sh+}
S_{h+}\left(E_{n},T_{n-1}\right)=\left(\begin{array}{cc}
    \mathrm{e}^{-\frac{h T^{\left(n-1\right)}_{n-1}}{2}} & b_{h+}^{\left(n\right)}  \\
    0 & \mathrm{e}^{-\frac{h e_{nn}}{2}}
  \end{array}\right),
  \end{equation}
where $T^{\left(n-1\right)}_{n-1}=\mathrm{diag}\left(\ell^{\left(n-1\right)}_1+n-2,...,\ell^{\left(n-1\right)}_{n-1}+n-2\right)$, $b^{\left(n\right)}_{h+}=\left(\big(b^{(n)}_{h+}\big)_1,...,\big(b^{(n)}_{h+}\big)_{n-1}\right)^{\intercal}$ and
\begin{align*}
\big(b^{(n)}_{h+}\big)_{j }=
-\frac{\left(h \mathrm{e}^{-\frac{\pi \mathrm{i}}{2}}\right)^{\frac{h}{2\pi \mathrm{i}}\left(\ell^{\left(n-1\right)}_j+n-2\right)}}{\left(h \mathrm{e}^{\frac{\pi \mathrm{i}}{2}}\right)^{\frac{h}{2\pi \mathrm{i}}\left(e_{n n}+1\right)}} \frac{\prod_{l=1,l\neq j}^{n-1} \Gamma\left(1+\frac{h}{2\pi \mathrm{i} }\left(\ell^{\left(n-1\right)}_j-1-\ell^{\left(n-1\right)}_l\right)\right)}{\prod_{l=1}^{n}\Gamma\left(1+\frac{h}{2\pi \mathrm{i} }\left(\ell^{\left(n-1\right)}_j-1-\ell^{\left(n\right)}_l\right)\right)} h \alpha^{(n-1)}_j.
\end{align*}
Similarly, the matrix $S_{h-}\left(E_n,T_{n-1}\right)$ is lower triangular and ignored here.
\end{thm}
\begin{proof}
On the sector $z\in S\left(-\pi,0\right)$, we have $F_{h,-1}\left(z;E_n,T_{n-1}\right)\cdot \mathrm{e}^{\frac{h\delta_{n-1}\left(T_{n-1}\right)}{2}} S_{h+}\left(E_n, T_{n-1}\right) =F_{h,0}\left(z;E_n,T_{n-1}\right)$. Then by Proposition \ref{mainpro}, we have 
\begin{equation}\label{Sh+ eq}
     S_{h+}\left(E_n,T_{n-1}\right) Y_h U_{h,0}=\mathrm{e}^{-\frac{h\delta_{n-1}\left(T_{n-1}\right)}{2}} Y_h U_{h,-1}.
\end{equation}
Plugging \eqref{explicit Sh+} into \eqref{Sh+ eq}, we find that $\big(b^{(n)}_{h+}\big)_j$ satisfies $n$ equations as follows
\begin{equation*}
    \mathrm{e}^{-\frac{h\left(\ell^{\left(n-1\right)}_j+n-2\right)}{2}} \alpha^{\left(n-1\right)}_j \left(U_{h,0}\right)_{j k}+\big(b^{(n)}_{h+}\big)_j \left(U_{h,0}\right)_{n k}=\mathrm{e}^{-\frac{h\left(\ell^{\left(n-1\right)}_j+n-2\right)}{2}} \alpha^{\left(n-1\right)}_j \left(U_{h,-1}\right)_{j k}, \quad 1\leq k \leq n,
\end{equation*}
which imply that
\begin{equation*}
    \big(b^{(n)}_{h+}\big)_{j}=\mathrm{e}^{-\frac{h\left(\ell^{\left(n-1\right)}_j+n-2\right)}{2}} \alpha^{\left(n-1\right)}_j\left(\left(U_{-1}\right)_{j k}-\left(U_{h,0}\right)_{j k}\right)\left(U_0\right)_{n k}^{-1}.
\end{equation*}
We can rewrite $\left(U_0\right)_{n k}^{-1}$ and $\left(U_{-1}\right)_{j k}-\left(U_{h,0}\right)_{j k}$ into the following forms in the same way as \eqref{U0 nk} and \eqref{U1-U0}
\begin{equation}\label{Uh0 nk}
    \left(U_0\right)_{n k}^{-1}=\frac{h}{2\pi \mathrm{i} }\frac{\pi \prod_{l=1,l\neq j}^{n-1}\Gamma\left(\alpha_{n k,l}\right)\left(\ell^{\left(n\right)}_k-\ell^{\left(n-1\right)}_j\right){\left(h\mathrm{e}^{\frac{\pi \mathrm{i} }{2}}\right)^{\alpha_{nk,n}}}}{\mathrm{sin}\left(\frac{h}{2 \mathrm{i} }\left(\ell^{\left(n\right)}_k-\ell^{\left(n-1\right)}_j\right)\right)\Gamma\left(1-\frac{h}{2\pi \mathrm{i} }\left(\ell^{\left(n\right)}_k-\ell^{\left(n-1\right)}_j\right)\right)\prod_{l=1,l \ne k}^{n} \Gamma\left(\beta_{nk,l}\right)},
\end{equation}
and 
\begin{equation}\label{Uh1-Uh0}
    \left(U_{-1}\right)_{j k}-\left(U_{h,0}\right)_{j k}=\frac{1}{\left(h \mathrm{e}^{\frac{\pi\mathrm{i} }{2}}\right)^{\alpha_{jk,j}}} \frac{-2\mathrm{i} \cdot\mathrm{sin}\left(\alpha_{j k,j}\pi\right)}{\ell^{\left(n\right)}_k-\ell^{\left(n-1\right)}_j}\frac{\prod_{l=1,l\ne k}^{n} \Gamma\left(\beta_{jk,l}\right)}{\prod_{l=1,l\ne j}^{n-1} \Gamma\left(\alpha_{jk,l}\right)}\frac{\prod_{l=1,l\ne j}^{n-1} \Gamma\left(\alpha_{jk,l}-\alpha_{jk,j}\right)}{\prod_{l=1,l\ne k}^{n}\Gamma\left(\beta_{jk,l}-\alpha_{jk,j}\right)}.
\end{equation}
Combining \eqref{Uh0 nk} and \eqref{Uh1-Uh0}, we have
\begin{align*}
    \big(b^{(n)}_{h+}\big)_{j}
  &=-\mathrm{e}^{-\frac{h\left(\ell^{\left(n-1\right)}_j+n-2\right)}{2}}h \alpha^{\left(n-1\right)}_j \left(h \mathrm{e}^{\frac{\pi \mathrm{i}}{2}}\right)^{\alpha_{n k,n}-\alpha_{j k,j}}\frac{\prod_{l=1,l\neq j}^{n-1} \Gamma\left(1+\frac{h}{2\pi \mathrm{i} }\left(\ell^{\left(n-1\right)}_j-\ell^{\left(n-1\right)}_l\right)\right)}{\prod_{l=1}^{n}\Gamma\left(1+\frac{h}{2\pi \mathrm{i} }\left(\ell^{\left(n-1\right)}_j-\ell^{\left(n\right)}_l\right)\right)}\\
  &=-\frac{\left(h \mathrm{e}^{-\frac{\pi \mathrm{i}}{2}}\right)^{\frac{h}{2\pi \mathrm{i}}\left(\ell^{\left(n-1\right)}_j+n-2\right)}}{\left(h \mathrm{e}^{\frac{\pi \mathrm{i}}{2}}\right)^{\frac{h}{2\pi \mathrm{i}}\left(e_{n n}+1\right)}} \frac{\prod_{l=1,l\neq j}^{n-1} \Gamma\left(1+\frac{h}{2\pi \mathrm{i} }\left(\ell^{\left(n-1\right)}_j-1-\ell^{\left(n-1\right)}_l\right)\right)}{\prod_{l=1}^{n}\Gamma\left(1+\frac{h}{2\pi \mathrm{i} }\left(\ell^{\left(n-1\right)}_j-1-\ell^{\left(n\right)}_l\right)\right)} h \alpha^{(n-1)}_j.
\end{align*}
\end{proof}

\subsection{Extension of the Stokes matrices of quantum system \eqref{introqeq}}\label{qsecextension}
Similar to $P_{n-1}$ and $Q_{n-1}$ in the classical setting, $P_{h,n-1}$ and $Q_{h,n-1}$ also bring singularity. However, for a real number $h$, the system \eqref{introqeq} is nonresonant and Stokes matrices are well defined, so 
\begin{equation*}
S_{h\pm }\left(E_n, T\right)=\operatorname{diag}\left({P}_{h,n-1},1\right)S_{h\pm}\left(E_n, T_{n-1}\right)\operatorname{diag}\left({P}_{h,n-1},1\right)^{-1}
\end{equation*}
always holds as the case in Proposition \ref{main stokes}.
\begin{thm}\label{T expli qstokes}
When $h$ is a real number, the Stokes matrix $S_{h +}\left(E_{n},T\right)$ of the system \eqref{introqeq} takes the block matrix form
\begin{equation*}
S_{h+}\left(E_{n},T\right)=\left(\begin{array}{cc}
    \mathrm{e}^{\frac{-h T^{\left(n-1\right)}}{2}} & b_{h+}  \\
    0 & \mathrm{e}^{\frac{-h e_{nn}}{2}}
  \end{array}\right),
  \end{equation*}
where $b_{h+}=\left(\left(b_{h+}\right)_1,...,\left(b_{h+}\right)_{n-1}\right)^{\intercal}$ with entries as follows
\begin{equation*}
   \left(b_{h+}\right)_k= -\sum_{i=1}^{n-1}\sum_{j=1}^{n-1} \frac{(h\mathrm{e}^{-\frac{\pi \mathrm{i}}{2}})^{\frac{h}{2\pi \mathrm{i} }(\ell^{(n-1)}_i+n-2)}}{(h\mathrm{e}^{\frac{\pi \mathrm{i}}{2}})^{\frac{h }{2\pi \mathrm{i} }\left(e_{n n}+1\right)}} \frac{\prod_{l\neq i}^{n-1} \Gamma\left(1+\frac{h}{2\pi \mathrm{i} }(\ell^{(n-1)}_i-1-\ell^{(n-1)}_l)\right)}{\prod_{l=1}^{n}\Gamma\left(1+\frac{h}{2\pi \mathrm{i} }(\ell^{(n-1)}_i-1-\ell^{(n)}_l)\right)} (\mathrm{Pr}_{i})_{k j} \cdot h e_{j n}.
\end{equation*}
\end{thm}

\begin{proof}
For any Gelfand-Tsetlin basis vector $\xi_{\Lambda}\in L\left(\lambda\right)$, for $1\leq j,k \leq n-2$, we have
\begin{align*}
    &\left[\left(\mathcal{P}_{h,n-1}\right)_{i j},\ell^{\left(n-1\right)}_k\right]\xi_{\Lambda}=\left[\left[e_{i,n-1},(-1)^{n-1+j}\Delta^{1,...,n-3,n-2}_{1,...,n-3,n-2}\left(T\big(\ell^{(n-1)}_j\big)\right)\right],\ell^{\left(n-1\right)}_k\right]\xi_{\Lambda}=0,\quad 1\leq i\leq n-2,\\
    &\left[\left(\mathcal{P}_{h,n-1}\right)_{n-1, j},\ell^{\left(n-1\right)}_k\right]=\left[(-1)^{n-1+j}\Delta^{1,...,n-2}_{1,...,n-2}\left(T\big(\ell^{(n-1)}_j\big)\right),\ell^{\left(n-1\right)}_k\right]=0.
\end{align*}
So for $1\leq i,j,k \leq n-1$, $
    \left[\left(\mathcal{P}_{h,n-1}\right)_{i j},\ell^{\left(n-1\right)}_k\right]=0$, together with the identity
$
    \left(b_{h+}\right)_k=\sum_{i=1}^{n-1}\left(P_{h,n-1}\right)_{k i} \big(b^{(n)}_{h+}\big)_i
    $, we have
\begin{align*}
    &(b_{h+})_k\\
    =&-\sum_{i=1}^{n-1}\sum_{j=1}^{n-1} \frac{(h \mathrm{e}^{-\frac{\pi \mathrm{i}}{2}})^{\frac{h}{2\pi \mathrm{i}}\left(\ell^{\left(n-1\right)}_i+n-2\right)}\prod_{l\neq i}^{n-1} \Gamma\left(1+\frac{h}{2\pi \mathrm{i} }(\ell^{(n-1)}_i-1-\ell^{(n-1)}_l)\right)}{(h \mathrm{e}^{\frac{\pi \mathrm{i}}{2}})^{\frac{h}{2\pi \mathrm{i}}\left(e_{n n}+1\right)}\prod_{l=1}^{n}\Gamma\left(1+\frac{h}{2\pi \mathrm{i} }(\ell^{(n-1)}_i-1-\ell^{(n)}_l)\right)}  (\mathcal{Q}_{h,n-1})_{ki} \frac{1}{(D_{h,n-1})_{i}} (\mathcal{P}_{h,n-1})_{ij} \cdot he_{jn}\\
=& -\sum_{i=1}^{n-1}\sum_{j=1}^{n-1} \frac{(h\mathrm{e}^{-\frac{\pi \mathrm{i}}{2}})^{\frac{h}{2\pi \mathrm{i} }(\ell^{(n-1)}_i+n-2)}}{(h\mathrm{e}^{\frac{\pi \mathrm{i}}{2}})^{\frac{h }{2\pi \mathrm{i} }\left(e_{n n}+1\right)}} \frac{\prod_{l\neq i}^{n-1} \Gamma\left(1+\frac{h}{2\pi \mathrm{i} }(\ell^{(n-1)}_i-1-\ell^{(n-1)}_l)\right)}{\prod_{l=1}^{n}\Gamma\left(1+\frac{h}{2\pi \mathrm{i} }(\ell^{(n-1)}_i-1-\ell^{(n)}_l)\right)} (\mathrm{Pr}_{i})_{k j} \cdot h e_{j n}.
\end{align*}
Here the last row in the above identity is derived from \eqref{Pk=PQ} and the singularity in the first row arising from $\left(D_{h,n-1}\right)_i$ in the denominator does not matter, since the $\left(k,j\right)$-entry of the projection operator $\mathrm{Pr}_{i}$ is without singularity.
\end{proof}

\end{document}